\documentclass[10pt]{smfart}

\RequirePackage[T1]{fontenc}
\RequirePackage{amsfonts,latexsym,amssymb}
\RequirePackage[frenchb]{babel}
\addto\extrasfrenchb{\bbl@nonfrenchitemize
\bbl@nonfrenchspacing}

\RequirePackage{mathrsfs}
\let\mathcal\mathscr

\RequirePackage{smfenum}

\theoremstyle{plain}
\numberwithin{equation}{section}
\newtheorem{prop}[equation]{\propname}
\newtheorem{theo}[equation]{\theoname}

\newtheorem{coro}[equation]{\coroname}

\newtheorem{lemm}[equation]{\lemmname}
\theoremstyle{definition}
\theoremstyle{remark}
\newtheorem{defi}[equation]{\definame}
\newtheorem{rema}[equation]{\remaname}

\usepackage[matrix,arrow]{xy}
\usepackage{url}
\usepackage{fge}
\newcommand{\moins}{\mathbin{\fgebackslash}}

\usepackage[pagebackref]{hyperref}

\let\cal\mathcal
\let\goth\mathfrak

\def\bpart{{\boldsymbol\partial}}

\def\Tp{{\mathbb T}(\R^+)} \def\Tc{{\mathbb T}(\C)} \def\Tr{{\mathbb T}(\R)}
\def\Tpo{{\mathbb T}^0(\R^+)} \def\Tco{{\mathbb T}^0(\C)} \def\Tro{{\mathbb T}^0(\R)}

\def\F{{F(\phi,{\cal B},y)}}

\def\Q{{\bf Q}} \def\Z{{\bf Z}}
\def\C{{\bf C}}
\def\N{{\bf N}}
\def\A{{\bf A}}
\def\R{{\bf R}}
\def\O{{\cal O}}
\def\dual{{\boldsymbol *}}
\def\bmu{{\boldsymbol\mu}}
\def\beps{{\boldsymbol\epsilon}}

\def\epsilon{\varepsilon}
\let\emptyset\varnothing

\makeatletter
\let\over\@@over
\makeatother

\begin{document}
\title[M\'ethode de Shintani]
{La m\'ethode de Shintani et ses variantes}
\author{Pierre Colmez}
\address{CNRS, IMJ-PRG, Sorbonne Universit\'e, 4 place Jussieu,
75005 Paris, France}
\email{pierre.colmez@imj-prg.fr}
\dedicatory{\`A la m\'emoire de Nicolas Bergeron}
\begin{abstract}
Nous donnons plusieurs variantes de la m\'ethode de Shintani
de d\'ecomposition en c\^ones simpliciaux d'un tore modulo
l'action d'un r\'eseau, et explorons quelques
applications \`a l'\'etude de fonctions~$L$ de Hecke aux entiers.
En particulier, nous prouvons la formule analytique du nombre de classes pour un corps
totalement r\'eel, directement en $s=0$.
\end{abstract}
\begin{altabstract}
We give several versions of Shintani's method for the
decomposition into simplicial cones of the fundamental domain
of a torus modulo a lattice, and we investigate
some applications to the study of Hecke $L$-functions at
integer points.
In particular, we prove the analytic class number formula for a totally real field,
directly at $s=0$.
\end{altabstract}
\setcounter{tocdepth}{1}

\maketitle
\stepcounter{tocdepth}
{\Small
\tableofcontents
}

\section*{Introduction}
Soit $F$ un corps totalement r\'eel.  Si $\chi$ est un caract\`ere de
Hecke de $F$, d'ordre fini, Hecke~\cite{hecke2} a prouv\'e que la fonction $L(\chi,s)$ associ\'ee
\`a $\chi$ admet un prolongement analytique \`a tout le plan complexe, holomorphe
en dehors d'un p\^ole simple en $s=1$ si $\chi=1$, et Klingen~\cite{klingen}, Siegel~\cite{siegel} ont prouv\'e
que les valeurs aux entiers~$\leq 0$ de $L(\chi,s)$ sont rationnelles (au sens que
$L(\chi,-k)\in\Q(\chi)$ et $L(\chi^\sigma,-k)=L(\chi,-k)^\sigma$ pour tout $\sigma\in{\rm Gal}(\Q(\chi)/\Q)$).

Shintani~\cite{Sh} a trouv\'e une m\'ethode alternative pour d\'emontrer ces deux r\'esultats; cette m\'ethode
repose sur l'existence de d\'ecompositions en c\^ones simpliciaux de domaines fondamentaux 
de $(\R\otimes F)_{>0}$ modulo
l'action du groupe $U_F^+$ des unit\'es totalement positives de $\O_F$ (d\'ecompositions de Shintani).

Dans cet article nous donnons une variante de la m\'ethode de Shintani qui permet de remplacer
$U_F^+$ par n'importe quel sous-groupe libre $V$ d'indice fini du groupe des unit\'es $U_F$ (pas
n\'ecessairement totalement positives).  Dans le cas o\`u $V\subset U_F^+$,
cette variante raffine la notion de {\og signed cone decomposition\fg}
de~\cite{DF2}.  
Elle permet d'\'etudier l'ordre d'annulation de $L(\chi,s)$ en un entier n\'egatif directement,
sans utiliser l'\'equation fonctionnelle reliant
$L(\chi,s)$ et $L(\chi^{-1},1-s)$ et les propri\'et\'es de la fonction $\Gamma$ d'Euler; on en
d\'eduit par exemple une preuve de la formule analytique du nombre de classes
\begin{equation}\label{classes}
\lim_{s\to 0}s^{1-[F:\Q]}\zeta_F(s)=-h_F\,{\rm Reg}(U_F)
\end{equation}
directement en $s=0$.

\subsection{D\'ecompositions de Shintani}
On note 
$${\rm Tr}:\R\otimes F\to\R\quad{\rm et}\quad {\rm N}:(\R\otimes F)^\dual\to\R^\dual$$
les applications trace et norme et $\epsilon:\R^\dual\to\{\pm1\}$ le caract\`ere signe.
Remarquons que, si ${\rm H}_F:={\rm Hom}(F,\R)$, alors
$$\R\otimes F\cong \R^{{\rm H}_F}=\prod_{\tau\in {\rm H}_F}\R^\tau$$
On note $(e_\tau)_\tau$ la base naturelle de $\R\otimes F$ fournie par ces identifications.

Une d\'ecomposition de Shintani usuelle (convenablement modifi\'ee pour ne faire intervenir
que des c\^ones de dimension~$n$) se traduit par une identit\'e
du genre (o\`u ${\bf 1}_X$ est la fonction caract\'eristique de $X$):
\begin{equation}\label{Si1}
{\bf 1}_{(\R\otimes F)_{>0}}=\sum_{v\in V}\sum_B {\bf 1}_{vC_B}
\end{equation}
 o\`u les $B$
qui interviennent sont des bases $(f_{B,1},\dots,f_{B,n})$ de $F$ sur $\Q$, et $C_B$ est un c\^one
dans $(\R\otimes F)_{>0}$ engendr\'e par les $f_{B,i}$.

On peut prendre la transform\'ee de Laplace des deux membres, et obtenir une identit\'e
(o\`u $(f_{B,1}^\vee,\dots,f_{B,n}^\vee)$ est la base de $F$ sur $\Q$ duale de $(f_{B,1},\dots,f_{B,n})$
pour la forme $(x,y)\mapsto {\rm Tr}\,xy$):
\begin{equation}\label{Si2}
\frac{1}{{\rm N}(z)}
=\sum_{v\in V}\sum_B\frac{\det B}{{\rm Tr}(f^\vee_{B,1}vz)\cdots {\rm Tr}(f^\vee_{B,n}vz)}
\end{equation} 
ce qui fournit une version des d\'ecompositions de Shintani qui a un sens
aussi dans le cas non totalement r\'eel (version utilis\'ee dans~\cite{Cz2,CS} pour
\'etudier les valeurs sp\'eciales de fonctions $L$ de Hecke d'extensions de corps quadratiques imaginaires;
remarquons qu'appliquer $\big(\prod_\tau\frac{\partial}{\partial z_\tau}\big)^{k-1}$ aux deux membres
forunit aussi une identit\'e pour $\frac{1}{{\rm N}(z)^k}$).

Si on reste dans le cadre totalement r\'eel, on peut prendre la transform\'ee de Fourier (par rapport \`a $t$,
en fixant $y$) des deux membres de (\ref{Si2})
 pour $z=t+iy$; on obtient alors une version de la
notion de d\'ecomposition de Shintani sous la forme
\begin{equation}\label{Si3}
\chi_{(e_\tau)_\tau}(x,y)=\sum_{v\in V}\sum_B{\rm N}(v)\chi_{vB}(x,y)
\end{equation} 
o\`u, si $y$ est fix\'e, $\chi_{B}(x,y)$ est\footnote{Par exemple, si ${\rm N}(y)\neq 0$,
alors $\chi_{(e_\tau)_\tau}(x,y)$ 
est $\pm$ la fonction caract\'eristique
du c\^one engendr\'e par les $y_\tau e_\tau$ (o\`u $\pm=\epsilon({\rm N}(y))$).}
 $\pm {\bf 1}_{C_B(y)}$, et
$C_B(y)$ est un c\^one engendr\'e par les $\pm f_{B,i}$ (o\`u les $\pm$ devant ${\bf 1}_{C_B(y)}$
et $f_{B,i}$ d\'ependent de $y$, cf.~\no\ref{SSS9}).

C'est cette derni\`ere version qui est d\'evelopp\'ee dans cet article.

\subsection{Application aux fonctions~$L$}
Le (ii) du th.\,\ref{Si6} et le (i) du th.\,\ref{Si7} ci-dessous
expliquent comment on utilise une d\'ecomposition de Shintani~(\ref{Si3}). Le reste
du th.\,\ref{Si7} s'en d\'eduit par des m\'ethodes standard mais un peu plus d\'elicates
que d'habitude car la convergence de la transform\'ee de Mellin de $F(\phi,{\cal B},y)$ n'est
pas tr\`es bonne.

Le choix d'un syst\`eme de repr\'esentants du groupe des classes d'id\'eaux de $F$
permet de d\'ecomposer $L(\chi,s)$ comme une somme de fonctions du type
$L(\phi,\psi,s)$, o\`u:

$\bullet$ $\psi:(\R\otimes F)^\dual\to\{\pm1\}$ est un caract\`ere continu,

$\bullet$ $\phi: F\to \C$ est constante modulo ${\goth a}$ et \`a support dans ${\goth b}$,
o\`u ${\goth a},{\goth b}$ sont des id\'eaux fractionnaires de $F$, et v\'erifie
$\phi(v\alpha)=\psi(v)\phi(\alpha)$ pour tous $\alpha\in F$ et $v\in U_F$.

$\bullet$ $L(\phi,\psi,s)=\sum_{\alpha\in F^\dual/U_F}\psi(\alpha)\phi(\alpha)|{\rm N}(\alpha)|^{-s}$.

\vskip2mm
Une d\'ecomposition du type~(\ref{Si2}) fournit une identit\'e du genre
$$\sum_{\alpha\in (\alpha_0+{\goth a})/V}\frac{1}{{\rm N}(\alpha)}=
\sum_B\sum_{\alpha\in \alpha_0+{\goth a}}
\frac{\det B}{{\rm Tr}(f^\vee_{B,1}v\alpha)\cdots {\rm Tr}(f^\vee_{B,n}v\alpha)}$$
et comme on somme sur un r\'eseau on voit, au moins formellement,
 appara\^{\i}tre des produits de $\pi {\rm cotg}\,\pi r$,
avec $r\in\Q$ (o\`u des s\'eries d'Eisenstein en des points de torsion si on part d'un corps
quadratique imaginaire).  On en d\'eduit des r\'esultats d'alg\'ebricit\'e
mais, comme les s\'eries ne convergent pas absolument, il faut
faire un petit peu attention.

\vskip2mm
La m\'ethode de Shintani traditionnelle part de l'identit\'e (pour $\alpha\in F_{>0}^\dual$):
$$\frac{1}{{\rm N}(\alpha)^s}=
\frac{1}{\Gamma(s)^n}\int_{(\R\otimes F)_{>0}}q^\alpha \prod_\tau y_\tau^{s-1}\,dy_\tau,
\quad{\text{o\`u $q^\alpha:=e^{-{\rm Tr}(\alpha y)}$}}$$
  Une d\'ecomposition modulo $U_F^+$ en c\^ones fournit
une identit\'e du type 
$$\sum_{\alpha\in F_{>0}^\dual/U_F^+}\phi(\alpha)q^\alpha=F(\phi,y)$$
o\`u $F(\phi,y)$ est ``\`a d\'ecroissance rapide'' \`a l'infini de $(\R\otimes F)_{>0}$,
et est une combinaison lin\'eaire de $\frac{q^\alpha}{(1-q^{f_1})\cdots(1-q^{f_n})}$
o\`u $\alpha,f_1,\dots,f_n\in F_{>0}^\dual$.
L'utilisation des d\'ecompositions~(\ref{Si3}) repose sur le m\^eme point de d\'epart, mais
il y a des diff\'erences sensibles pour le reste des arguments.

\begin{theo}\phantomsection\label{Si6}
Soient $\psi,\phi$ comme ci-dessus, soit $\Lambda$ le $\Z$-module\footnote{Il est de type fini
car ${\goth b}/{\goth a}$ est un ensemble fini.}
engendr\'e par les $\phi(\alpha)$, soit $V$ un sous-groupe libre d'indice fini de
$U_F$, et soit ${\cal B}=\sum n_BB$ v\'erifiant l'identit\'e~{\rm (\ref{Si3})}.

{\rm (i)}
Il existe 
$$F(\phi,{\cal B},y)\in\Lambda\big[q^\alpha,\tfrac{1}{1-q^\alpha},\ \alpha\in F^\dual\big]$$
tel que, 
pour tout $y$ qui n'est pas un p\^ole de $F(\phi,{\cal B},y)$,
 $$\sum_{\alpha\in F}\phi(\alpha)\sum_B n_B\chi_{B}(\alpha,y)q^\alpha=F(\phi,{\cal B},y)$$

{\rm (ii)}
Soit
$$F^\dual(\phi,{\cal B},y)=F(\phi,{\cal B},y)-\phi(0)\chi_{\cal B}(0,y)$$
Pour presque tout $y\in (\R\otimes F)^\dual$
$$\sum_{v\in V}\psi(v){\rm N}(v)F^\dual(\phi,{\cal B},vy)=
\epsilon({\rm N}(y))
\hskip-5mm \sum_{{\beta\in F^\dual }\atop{\beta^{-1}y\in(\R\otimes F)_{>0}}}
\phi(\beta)e^{-{\rm Tr}(\beta y)}.$$
\end{theo}

\begin{rema}\phantomsection\label{Si11}
Une recette de Cassou-Nogu\`es~\cite{CN2} permet de lisser $\phi$ pour que $F(\phi,{\cal B},y)$
devienne ${\cal C}^\infty$ sur $\R^n$ sans trop modifier $L(\phi,\psi,s)$ (cf.~rem.\,\ref{smoo2.1}),
ce qui permet de supposer que $F(\phi,{\cal B},y)$ est ${\cal C}^\infty$ sur $\R^n$ pour les
applications aux fonctions~$L$.
\end{rema}

\begin{theo}\phantomsection\label{Si7}
Supposons que $F(\phi,{\cal B},y)$
est ${\cal C}^\infty$ sur $\R^n$.

{\rm (i)}
Si $0<{\rm Re}(s)<\frac{1}{n}$,
$$L(\phi,\psi,s)={1\over{[U_F:V]}}{1\over{\Gamma(s)^n}}\int_{\R\otimes F}
F(\phi,{\cal B},y)\psi(y)\epsilon({\rm N}(y))|{\rm N}(y)|^{s-1}dy$$

{\rm(ii)}
$L(\phi,\psi,s)$ admet un prolongement analytique \`a tout
le plan complexe.

{\rm (iii)}  
Si $k\in{\bf N}$, soit $I_k=\{\tau\in{\rm H}_F,\ \psi_{|\R^\tau}=\epsilon^k\}$.
Alors, si $k\neq0$ ou si 
$\phi(0)=0$, la fonction $L(\phi,\psi,s)$ a un
z\'ero d'ordre~$\geq |I_k|$ en $s=-k$ et
$$\lim_{s\rightarrow -k}(s+k)^{-|I_k|}L(\phi,\psi,s)=
{{2^{n-|I_k|}}\over{[U_F:V]}}\int_{{\bf R}^{I_k}}
\nabla^k F(\phi,{\cal B},y)\prod_{\tau \in I_k}{{dy_\tau }\over{y_\tau }},$$
o\`u $\nabla=\prod_\tau{\partial\over{\partial y_\tau}}$ et
${\bf R}^{I_k}$ est le sous-${\bf R}$-espace vectoriel de dimension
$|I_k|$ d'\'equation $y_\tau=0$ si $\tau\notin I_k$.
\end{theo}

\begin{rema}\phantomsection\label{Si12}
Si $k=0$ et $\phi(0)\neq 0$, il faut tenir compte du terme $\phi(0)\chi_{\cal B}(0,y)$ qui vaut $\pm \phi(0)$
sur un certain c\^one, et contribue un terme ayant un z\'ero d'ordre $n-1$ au lieu du $n$ attendu,
le coefficient du terme dominant \'etant le volume de l'intersection
de ce c\^one et de l'hyperbolo\"{\i}de $|{\rm N}(z)|=1$ (c'est de l\`a que sort le facteur
${\rm Reg}(U_F)$ de la formule~(\ref{classes})).
\end{rema}

\begin{rema}\phantomsection\label{Si13}
Une des motivations pour permettre d'utiliser un sous-groupe libre quelconque de $U_F$
(et pas seulement un sous-groupe de $U_F^+$) \'etait de red\'emontrer les divisibilit\'es
par $2^{[F:\Q]}$ de Deligne-Ribet~\cite{DR} (\`a part dans le cas exceptionnel o\`u il y a
divisibilit\'e par exactement $2^{[F:\Q]-1}$).

Cassou-Nogu\`es~\cite[cor.\,18]{CN1} a obtenu des r\'esultats dans cette direction sous
des conditions du type ``conjecture de Leopolodt''. Le th.\,\ref{Si7}
permet de prouver inconditionnellement 
une divisibilit\'e par $2^{[F:\Q]-1}$ (th.\,\ref{heck12}) et, quand on peut
\'ecrire $\phi(\alpha)$ sous la forme $\phi_1(\alpha)\pm\phi_1(-\alpha)$ (ce qui est souvent le cas,
mais pas toujours), on peut en d\'eduire une divisibilit\'e par $2^{[F:\Q]}$.  Il semble toutefois
difficile de retrouver l'int\'egralit\'e des r\'esultats de~\cite{DR} par ces m\'ethodes.
\end{rema}


\subsection{C\'eoukonfaikoi}
Cet article comporte trois chapitres.

--- Le premier chapitre est consacr\'e aux d\'ecompositions de Shintani; la version~(\ref{Si1}) fait l'objet
du \S\,\ref{SSS3}, la version~(\ref{Si2}) fait l'objet du \S\,\ref{SSS6},
et la version~(\ref{Si3}) fait l'objet du \S\,\ref{SSS8}.

--- Le second chapitre est consacr\'e \`a la preuve du th.\,\ref{Si6} et \`a des compl\'ements
sur le comportement de $F(\phi,{\cal B},y)$ \`a l'infini
en vue de prouver la convergence de sa transform\'ee de Mellin.

--- Le trois\`eme chapitre est consacr\'e \`a la preuve du th.\,\ref{Si7} et \`a ses applications
aux fonctions~$L$ de Hecke (la formule analytique du nombre de classes en $s=0$
se trouve dans le th.\,\ref{heck4}).

\section{D\'ecompositions de Shintani}\label{SSS1}

\subsection{Notations}\label{SSS2}
Si ${\bf K}={\bf R}$ ou ${\bf C}$, et si $n$ est un entier~$\geq 1$,
on note:

$\bullet$ $e_1,\dots,e_n$ la base canonique de ${\bf K}^n$.

$\bullet$ ${\rm Tr}:{\bf K}^n\to {\bf K}$
et~${\rm N}:{\bf K}^n\to {\bf K}$
les applications
$${\rm Tr}(x)=\sum_{i=1}^nx_i\hskip.2cm{\rm et}\hskip.2cm
{\rm N}(x)=\prod_{i=1}^nx_i,\quad{\text{si $x=x_1e_1+\cdots+x_ne_n$.}}$$

$\bullet$
$\Tp$, $\Tr$, $\Tc$ les groupes $(\R_+^\dual)^n$,
$(\R^\dual)^n$, $(\C^\dual)^n$ et
$\Tpo$, $\Tro$, $\Tco$ leurs intersections
avec $\{z,\ {\rm N}(z)=1\}$.


$\bullet$ ${\bf 1}_X$ la fonction caract\'eristique
de $X\subset{\bf K}^n$.

$\bullet$ ${\rm Log}:(\C^\dual )^n\rightarrow \R^n$, l'application
$(z_1,\dots,z_n)\mapsto (\log|z_1|,\dots, \log|z_n|)$.

$\bullet$ $\Vert\ \Vert$ la norme $\Vert x \Vert=\sup_{1\leq i\leq n}|x_i|$.

\vskip2mm
Un {\it r\'eseau $V$ de $\Tc$} est un sous-groupe de $\Tc$
tel que ${\rm Log}$
induise un isomorphisme de $V$ sur un r\'eseau de l'hyperplan $H$
des \'el\'ements de trace $0$
(ceci implique en particulier que $V$ est libre, de rang $n-1$ sur $\Z$, et inclus
dans $\{z,\ |{\rm N}(z)|=1\}$).

\vskip2mm
Enfin, on note:

$\bullet$ 
$\epsilon(\sigma)\in\{\pm1\}$ la signature de $\sigma\in S_{n-1}$
(groupe des permutations de $\{1,\dots,n-1\}$),

$\bullet$ $\epsilon(y)\in\{\pm1\}$ le signe de $y\in\R^\dual$.

$\bullet$ $\epsilon(B)=\epsilon(\det B)$ si $B\in{\rm GL}_n(\R)$, et
$\epsilon(B)=0$ si $B\in{\rm M}_n(\R)$ et $\det B=0$.

\Subsection{Le cas r\'eel totalement positif}\label{SSS3}
Soit $V$ un r\'eseau de $\Tp$.
La m\'ethode de Shintani originelle~\cite{Sh} fournit 
un domaine fondamental de $\Tp$ modulo $V$ constitu\'e d'une r\'eunion
finie de c\^ones simpliciaux ouverts de dimensions variables. Nous donnons ci-dessous
une variante qui fournit un domaine fondamental
constitu\'e d'une r\'eunion
finie de c\^ones simpliciaux, tous de dimension $n$ mais qui ne sont ni ouverts ni ferm\'es
(chacun est obtenu en enlevant certaines faces \`a un c\^one ferm\'e).
L'existence de tels domaines fondamentaux a \'et\'e aussi observ\'ee par
Charollois, Dasgupta, Greenberg~\cite{CD}, 
et par Diaz y Diaz, Friedmann~\cite{DF2}.

\subsubsection{D\'ecompositions de Shintani effectives}\label{eff1}
Si $B=(f_1,\dots,f_n)$ est une base de ${\bf R}^n$ telle
que $e_n$ n'appartient \`a aucun des hyperplans engendr\'es par $n-1$
des vecteurs $f_1,\dots,f_n$,
soit $C_B^0$ le c\^one ouvert engendr\'e par $f_1,\dots,f_n$,
et soit $C_B$ le c\^one dont la fonction caract\'eristique
est donn\'ee par la formule
$${\bf 1}_{C_B}(x)=
\lim_{t\rightarrow 0^+}{\bf 1}_{C_B^0}(x+te_n).$$
Notons ${\bf R}(+1)={\bf R}_{\geq 0}$ et ${\bf R}(-1)={\bf R}_{<0} $.
Alors 
$$C_B={\bf R}(\epsilon_{B,1})f_1+
\cdots+{\bf R}(\epsilon_{B,n})f_n$$
 o\`u $\epsilon_{B,i}=1$ (resp. $\epsilon_{B,i}=-1$)
si  $e_n$ et $f_i$ sont (resp. ne sont pas)
du m\^eme c\^ot\'e de l'hyperplan engendr\'e par
$f_1,\dots,f_{i-1},f_{i+1},\dots,f_n$.

Une {\it d\'ecomposition de Shintani 
effective de $\Tp$ modulo $V$} est un
ensemble fini ${\cal B}$ de bases de ${\bf R}^n$ constitu\'ees 
d'\'el\'ements de $\Tp$ v\'erifiant les conditions suivantes:

(i)  Si $B=(f_1,\dots,f_n)\in{\cal B}$, alors $e_n$ n'appartient \`a
aucun des hyperplans engendr\'es par $n-1$ des vecteurs $f_1,\dots,f_n$,

(ii) Si $x\in \Tp$, il existe un et un seul couple $(B,v)\in{\cal B}\times V$
tel que $vx\in C_B$; autrement dit la r\'eunion disjointe des $C_B$, pour
$B\in{\cal B}$, forme un domaine fondamental de $\Tp$ modulo l'action de $V$,
ce qui peut aussi se r\'e\'ecrire sous la forme
$${\bf 1}_{\Tp}=\sum_{v\in V}\sum_{B\in{\cal B}}{\bf 1}_{v C_B}$$

Plus g\'en\'eralement une {\it d\'ecomposition de Shintani
de $\Tp$ modulo $V$} est un
ensemble fini ${\cal B}$ de couples $(B,n_B)$ o\`u $B$ est une base de ${\bf R}^n$ constitu\'ee
d'\'el\'ement de $\Tp$ v\'erifiant la condition (i) ci-dessus, et $n_B\in\Z$,
tels que
$${\bf 1}_{\Tp}=\sum_{v\in V}\sum_{B\in{\cal B}}n_B{\bf 1}_{v C_B}$$
(Si tous les $n_B$ valent $1$, on retombe sur une d\'ecomposition
effective, et si $n_B\in\{\pm 1\}$ pour tout $B$, on retombe
sur les {\og signed fundamental domains\fg} de~\cite{DF2}.)

L'ensemble des d\'ecompositions de Shintani
de $\Tp$ modulo $V$ est not\'e ${\rm Shin}(\Tp/V)$
et celui des d\'ecompositions effectives est not\'e ${\rm Shin}^{>0}(\Tp/V)$.

\subsubsection{Existence de d\'ecompositions de Shintani}\label{eff2}
On suppose $v=(v_1,\dots,v_n)\mapsto v_n$ injectif sur $V$.
\begin{rema}\phantomsection\label{S3}
Dans les applications que nous avons en vue,
$V$ est un sous-groupe
du groupe des unit\'es de l'anneau des entiers d'un corps totalement r\'eel $F$;
dans ce cas,
l'application $v\mapsto v_n$
est la restriction d'un plongement de $F$ dans ${\bf R}$ et est donc bien
injective.
\end{rema}
\begin{prop}\phantomsection\label{S2}
${\rm Shin}^{>0}(\Tp/V)$ n'est pas vide.
\end{prop}
\begin{proof}
   \`A quelques d\'etails pr\`es, la d\'emonstration
est la d\'emonstration originelle de Shintani~\cite{Sh}.
Soit ${\cal C}=\{x\in\Tp\ |\ {\rm Tr}(x)<{\rm Tr}(vx)\ \forall v\in V\moins\{1\}\,\}$.

$\bullet$ Il existe $\delta>0$ tel que $x_i\geq \delta\,{\rm Tr}(x)$
pour tous $x=(x_1,\dots,x_n)\in{\cal C}$ et $1\leq i\leq n$.
En effet, l'hypoth\`ese selon laquelle ${\rm Log}^0(V)$ est
un r\'eseau de $H$ implique qu'il existe
$v_i=(v_{i,1},\dots,v_{i,n})\in V$ v\'erifiant $v_{i,j}\leq\frac{1}{2}$ si
$j\neq i$, pour tout $1\leq i\leq n$ (on a alors $v_{i,i}\geq 2^{n-1}$).
L'in\'egalit\'e ${\rm Tr}(v_ix)>{\rm Tr}(x)$, qui se traduit par
 $(v_{i,i}-1)x_i>\sum_{j\neq i}(1-v_{i,j})x_j$, implique
alors $x_i>\frac{{\rm Tr}(x)}{2v_{i,i}-3}$, et on peut donc prendre
$\delta=\inf_i\frac{1}{2v_{i,i}-3}$.

$\bullet$ Si $x_i\geq \delta\,{\rm Tr}(x)$ pour tout
$1\leq i\leq n$, alors ${\rm Tr}(vx)>{\rm Tr}(x)$ si
$v=(v_1,\dots,v_n)$ v\'erifie $\sup v_i\geq\frac{n}{\delta}$.

Il r\'esulte des deux points pr\'ec\'edents que
si $W=\{v\in V\moins\{1\},\ \sup v_i\leq \frac{n}{\delta}\}$ 
(cet ensemble est fini), alors
 ${\cal C}=\{x\in\Tp\ |\ {\rm Tr}(x)<{\rm Tr}(vx)\ \forall v\in W\}$ et donc 
 ${\cal C}$ est un c\^one poly\'edral convexe.  D'autre part, ${\cal C}$
 n'est pas loin d'\^etre un domaine fondamental de $\Tp$ modulo l'action
 de $V$.
Plus pr\'ecis\'ement, la r\'eunion $\cup_{v\in V}v{\cal C}$ est une r\'eunion
disjointe et $\Tp\moins\cup_{v\in V}v{\cal C}$ est inclus dans la r\'eunion
(d\'enombrable) des hyperplans d'\'equations ${\rm Tr}(vx)={\rm Tr}(vwx)$,
pour $v,w\in V$.

Subdivisons ${\cal C}$ en c\^ones simpliciaux ouverts de telle sorte
que $e_n$ ne soit sur aucun des hyperplans
support\'es par les faces des c\^ones de dimension $n$
apparaissant dans la subdivision (cela est possible si et
seulement si $e_n$ n'est
sur aucun des hyperplans support\'es par les faces de ${\cal C}$; or un
tel hyperplan est d'\'equation ${\rm Tr}(vx)={\rm Tr}(x)$ avec $v\in W$ et ne contient
pas $e_n$ car on a suppos\'e que le morphisme "$n$-i\`eme coordonn\'ee"
est injectif).  Si $C$ est un c\^one de dimension $n$ apparaissant
dans la subdivision, choisissons une base $B(C)=(f_{1,C},\dots,f_{n,C})$
telle que $C$ soit le c\^one ouvert engendr\'e par $f_1,\dots,f_n$ 
(autrement dit, $C=C_{B(C)}^0$) et
notons ${\cal B}$ l'ensemble de ces bases.
Pour conclure, il suffit de prouver que
${\cal B}\in{\rm Shin}(\Tp/V)$.

   Tout d'abord, on remarque que la r\'eunion
$\cup_{B\in{\cal B}}\cup_{v\in V} vC_B^0$ est une r\'eunion disjointe.
D'autre part, localement,
$X=\Tp\moins
\cup_{B\in{\cal B}}\cup_{v\in V} vC_B^0$ est inclus dans une r\'eunion
finie d'hyperplans ne contenant pas $e_n$ (car si un hyperplan $H_0$ ne
contient pas $e_n$ et si $v\in\Tp$, alors $vH_0$ ne contient pas $e_n$).
On en tire $\lim_{t\rightarrow0^+}{\bf 1}_X(x+te_n)=0$ pour tout $x\in\Tp$.
D'autre part on a, si $x\in\Tp$,
$${\bf 1}_{\Tp}(x)=\lim_{t\rightarrow0^+}{\bf 1}_{\Tp}(x+te_n)=\lim_{t\rightarrow0^+}
\Bigl({\bf 1}_X(x+te_n)+\sum_{v\in V}\sum_{B\in{\cal B}}{\bf 1}_{vC_B^0}(x+te_n)\Bigr).$$
La somme dans le membre de droite \'etant localement finie, on peut \'echanger
les signes $\sum$ et $\lim_{t\rightarrow0^+}$ et on obtient
$${\bf 1}_{\Tp}(x)=\sum_{v\in V}\sum_{B\in{\cal B}}{\bf 1}_{C_{vB}}(x)
=\sum_{v\in V}\sum_{B\in{\cal B}}{\bf 1}_{vC_B}(x),$$
ce qui permet de conclure.
\end{proof}

\begin{rema}\phantomsection\label{sigma}
On peut construire des d\'ecompositions de Shintani
en partant d'une base de $V$:
soient $\eta_1,\dots,\eta_{n-1}$ une base de $V$ sur $\Z$.
Si $\sigma\in S_{n-1}$, posons $f_{1,\sigma}=1$ et $f_{i,\sigma}=
\prod_{j<i}\eta_j$ si $2\leq i\leq n$. 
Soit $B_\sigma=(f_{1,\sigma},\dots,f_{n,\sigma})$.

(i) Il r\'esulte 
de~\cite[Lemme\,2.2]{Cz2} et des arguments de la preuve ci-dessus pour faire dispara\^{\i}tre
les c\^ones de dimension~$<n$ en les incluant dans des faces bien choisies des
c\^ones de dimension~$n$, que 
$$\{{B_\sigma},\,\sigma\in S_{n-1}\}\in {\rm Shin}^{>0}(\Tp/V),\quad
{\text{si $\epsilon(\sigma)\epsilon(B_\sigma)$ 
ne d\'epend pas de $\sigma$.}}$$

(ii) Dans tous les cas~\cite{DF2},
$$\big\{\big(B_\sigma, \epsilon({\rm Log}(B_\sigma))\epsilon(B_\sigma)\big),\,\sigma\in S_{n-1}\big\}
\in {\rm Shin}(\Tp/V)$$
o\`u
${\rm Log}(B_\sigma)=(1,{\rm Log}\,f_{2,\sigma},\dots,{\rm Log}\,f_{n,\sigma})$
(pour faire le lien avec le (i), remarquons que $\epsilon({\rm Log}(B_\sigma))=\epsilon(\sigma)
\epsilon({\rm Log}(B_{\rm id}))$).
Voir le cor.\,\ref{R3} et la rem.\,\ref{R4} pour une g\'en\'eralisation de cet \'enonc\'e.
\end{rema}

\begin{rema}
Si $\sigma\in S_{n-1}$, notons $X_\sigma$ l'intersection
de $C_{B_\sigma}$ et $\Tpo$.  Alors
$\sum_{\sigma}\epsilon(\sigma)\epsilon(\det B_\sigma) X_\sigma$ est (au signe pr\`es) un domaine
fondamental de $\Tpo$ modulo $V$, ce qui se traduit par l'identit\'e
$$\sum_{\sigma}\epsilon(\sigma)\epsilon(\det B_\sigma) \int_{X_\sigma}d^\dual y=
\tfrac{1}{n}\det(1,{\rm Log}\,\eta_1,\dots,{\rm Log}\,\eta_{n-1})$$
(Si on fait la somme des lignes de la matrice, on obtient $(n,0,\dots,0)$ et on tombe
sur $n$ fois le r\'egulateur sous la forme standard (au signe pr\`es qui compense le signe
devant le domaine fondamental), ce qui explique le $\frac{1}{n}$ dans la formule ci-dessus.)

Les deux membres sont des p\'eriodes et,
si on croit \`a la conjecture de Konsevitch-Zagier, on devrait pouvoir \'etablir l'identit\'e
ci-dessus en n'utilisant que des op\'erations \'el\'ementaires. Pour $n=2$, le r\'esultat est imm\'ediat,
mais d\'ej\`a pour $n=3$, une telle d\'erivation ne semble pas imm\'ediate (le domaine fondamental 
naturel -- i.e.~le
parall\`el\'epip\`ede de sommets $\sum_i \delta_i\cdot {\rm Log}\,\eta_i$ pour $(\delta_i)_i\in\{0,1\}^{n-1}$ --
dans l'espace logarithmique ne se rel\`eve pas en un ensemble semi-alg\'ebrique dans $\Tpo$).
\end{rema}

\Subsection{Convergence de s\'eries}\label{SSS4}
Soit $V$ un r\'eseau de $\Tc$.
\begin{lemm}\phantomsection\label{D19.1}
Si $\delta>0$, alors $\sum_{v\in V}\Vert v\Vert^{-\delta}<\infty$.
\end{lemm}
\begin{proof}
Deux normes sur un espace vectoriel \'etant
\'equivalentes, 
si $\eta_1,\dots,
\eta_{n-1}$ est une base de $V$, il existe $c>0$ tel que
$\Vert {\rm Log}(\eta_1^{k_1}\dots\eta_{n-1}^{k_{n-1}})\Vert>
c(|k_1|+\cdots+|k_{n-1}|)$.  Par ailleurs, si $v\in V$, alors
$\Vert v\Vert\geq \exp{{1\over n}\Vert{\rm Log}(v)\Vert}$,
et donc $\Vert {\rm Log}(\eta_1^{k_1}\dots\eta_{n-1}^{k_{n-1}})\Vert^{-\delta}\leq
e^{-\delta'(|k_1|+\cdots+|k_{n-1}|)}$, avec $\delta'=\frac{1}{n}c\delta>0$.
Le r\'esultat s'en d\'eduit.
\end{proof}
\begin{lemm}\phantomsection\label{D19}
Soit ${\bf K}=\R,\C$.
  Soient $(g_i)_{i\in{\bf N}}$, une famille d'\'el\'ements
non nuls de ${\bf K}^n$ et $(a_i)_{i\in{\bf N}}$ une famille de r\'eels
positifs telle que $\sum_{i=0}^{+\infty}a_i<+\infty$.  
Alors
L'ensemble
des $y\in{\bf C}^n$ tels que l'on ait $|{\rm Tr}(g_iy)|<a_i\Vert g_i\Vert$
pour une infinit\'e de $i\in{\bf N}$ est de mesure nulle.
\end{lemm}
\begin{proof}
   Soit $H_i$ l'hyperplan d'\'equation ${\rm Tr}(g_iy)=0$.
Alors $|{\rm Tr}(g_iy)|<a_i\Vert g_i\Vert$ si et seulement si $d(y,H_i)<a_i$.
L'intersection de l'ensemble $Y_i$ des $y\in{\bf R}^n$ v\'erifiant
$|{\rm Tr}(g_i y)|<a_i\Vert g_i\Vert$ avec la boule $B(0,R)$ de centre 0 et de rayon
$R$ est donc de volume major\'e par $C_e(R^{n-1}a_i)^e$, avec $e=1,2$ si ${\bf K}=\R,\C$,
 pour une certaine
constante $C_e$.  D'autre part, l'intersection de $Y$ avec $B(0,R)$ est incluse,
pour tout $k\in{\bf N}$, dans l'ensemble $\cup_{i\geq k}(Y_i\cap B(0,R))$ dont
le volume peut \^etre rendu arbitrairement petit en prenant $k$ assez grand
car la s\'erie $\sum_{i=0}^{+\infty}a_i$ converge (c'est l'argument du lemme de Borel-Cantelli). 
Ceci permet de conclure.
\end{proof}

\begin{lemm}\phantomsection\label{D20}  
{\rm (i)}
Pour presque tout\footnote{Dans cet article, {\og presque tout\fg} r\'ef\`ere \`a la mesure de Lebesgue,
et ne signifie pas {\og \`a l'exception d'un nombre fini\fg}.}
$y\in {\bf K}^n$, l'ensemble des $v\in V$ tels que
$|{\rm Tr}(vy)|
\leq C\Vert v\Vert^{{1\over 2}}$ est fini.

{\rm (ii)}  Si $H$ est un hyperplan de ${\bf K}^n$, pour presque tout
$y\in H$, l'ensemble des $v\in V$ tels que
$|{\rm Tr}(vy)|
\leq C\Vert v\Vert^{{1\over 2}}$ est fini.
\end{lemm}
\begin{proof}
Le (i) est une cons\'equence directe des lemmes~\ref{D19} et~\ref{D19.1}.

Pour prouver le (ii), on d\'ecoupe $V$ suivant l'indice $i$ r\'ealisant le minimum
de $|v_i|$; on a alors $V=\cup_{i=1}^nV_i$, et il suffit de prouver
le r\'esultat pour $v\in V_i$ au lieu de $v\in V$.  Par sym\'etrie, on peut supposer $i=n$.
On peut mettre l'\'equation de $H$ sous la forme $y_n=\sum_{i=1}^{n-1}a_iy_i$ et alors
${\rm Tr}(yv)=y_1(v_1+a_1v_n)+\cdots+y_{n-1}(v_{n-1}+a_{n-1}v_n)$.
Soit $H_v$ l'hyperplan de $H$ d'\'equation ${\rm Tr}(yv)=0$, 
et soit $w=(v_1+a_1v_n,\dots,v_{n-1}+a_{n-1}v_n)$.  Alors $|{\rm Tr}(yv)|\leq \Vert v\Vert^{1/2}$
\'equivaut \`a $d(y,H_v)\leq \frac{\Vert v\Vert^{1/2}}{\Vert w\Vert}$.  Or,
comme $v\in V_n$, $|v_n|$ r\'ealise le minimum des $|v_i|$ et, en dehors d'un nombre fini
de $v$, on a $\Vert w\Vert\geq \frac{1}{2}\Vert v\Vert$ (on peut remplacer $\frac{1}{2}$ par
$1-\epsilon$).  On conclut en remarquant que $\sum_v\Vert v\Vert^{-1/2}<\infty$
et en utilisant le lemme~\ref{D19}.
\end{proof}

\begin{lemm}\phantomsection\label{D20.3}
Soit $B=(f_1,\dots,f_n)$ une base de $\C^n$ appartenant \`a ${\cal U}$.

{\rm (i)}  La s\'erie $\sum_{v\in V}\frac{1}{\prod_{i=1}^n\langle f_i,vz\rangle}$
converge pour presque tout $z\in\Tc$.

{\rm (ii)} La s\'erie converge dans l'espace ${\cal S}'(\C^n)$
des distributions temp\'er\'ees sur $\C^n$.
\end{lemm}
\begin{proof}

Le (i) est une cons\'equence imm\'ediate des lemmes~\ref{D20} et~\ref{D19.1}.

Pour prouver le (ii), il suffit de prouver que
la s\'erie $\sum_{v\in V}\int_{\C^n}\frac{\phi(z)}{\prod_{i=1}^n\langle f_i,vz\rangle}$ converge pour tout $\phi$ dans l'espace de Schwartz ${\cal S}(\C^n)$
de $\C^n$.  Soit donc $\phi\in {\cal S}(\C^n)$; il existe $C(\phi)>0$ tel
que $(1+\sum_{i=1}^n|z_i|)^{n+1}|\phi(z)|\leq C(\phi)$, pour tout $z\in\C^n$.
Comme $|{\rm N}(v)|=1$, on a
\begin{align*}
\big|\int_{\C^n}\frac{\phi(z)}{\prod_{i=1}^n\langle f_i,vz\rangle}\big|&=
\big|\int_{\C^n}\frac{\phi(v^{-1}z)}{\prod_{i=1}^n\langle f_i,z\rangle}\big|\\
&\leq \int_{\C^n}\frac{C(\phi)}{(1+\sum_{i=1}^n|v_i^{-1}z_i|)^{n+1}}
\,\frac{1}{\prod_{i=1}^n|\langle f_i,z\rangle|}:=C(\phi)R_I(v)
\end{align*}

On fixe $\epsilon>0$.
Si $I\subset\{1,\dots,n\}$, on note $U_I$ l'ensemble des
$z\in\C^n$ v\'erifiant $|\langle f_i,z\rangle|\leq \Vert v\Vert^\epsilon$
si $i\in I$ et $|\langle f_i,z\rangle|> \Vert v\Vert^\epsilon$
si $i\in I^c:=\{1,\dots,n\}\moins I$.  
On a alors $\int_{\C^n}=\sum_{I}\int_{U_I}$.

Si $I\neq\{1,\dots,n\}$,
on fait, sur $U_I$, le changement de variable $y_i=\langle f_i,z\rangle$,
pour $i\in I$ et $y_i=v_{j(i)}^{-1}z_{j(i)}$ (c'est un changement de
variable, pour tout choix injectif des $j(i)$, gr\^ace \`a l'hypoth\`ese $B\in{\cal U}(\C)$),
 o\`u les $j(i)$ pour $i\in I^c$,
sont choisis de telle sorte que $\prod_{i\in I^c}|v_{j(i)}|$ soit
minimal (ce produit est alors~$\leq \Vert v\Vert^{-\delta_I}$,
avec $\delta_I=\frac{|I|}{n-1}$ si $|I|<n$ -- si $|I|=n$ ce produit vaut $1$).
En majorant $\frac{1}{|\langle f_i,z\rangle|}$ par $\Vert v\Vert^{-\epsilon}$
si $i\in I^c$, et 
$\frac{1}{(1+\sum_{i=1}^n|v_i^{-1}z_i|)^{n+1}}$ par $
\frac{1}{(1+\sum_{i\in I^c}|y_i|)^{n+1}}$,
on obtient la majoration
$$R_I(v)\leq
\Vert v\Vert^{-(n-|I|)\epsilon}
\int_{V_I}
\big(\prod_{i\in I}\frac{1}{|y_i|}\big)
\frac{1}{(1+\sum_{i\in I^c}|y_i|)^{n+1}}\prod_{i\in I^c}|v_{j(i)}|^2,$$
o\`u $V_I=D(0,\Vert v\Vert^\epsilon)^I\times\C^{I^c}$ et $\prod_{i\in I^c}|v_{j(i)}|^2$
est le jacobien du changement de variable.
Comme $\int_{D(0,R)}\frac{1}{|y|}=2\pi R$ (passer en coordonn\'ees polaires),
 on en d\'eduit
l'existence de constantes $C_I$ telles que
$$\big|\int_{U_I}\frac{\phi(z)}{\prod_{i=1}^n\langle f_i,vz\rangle}\big|
\leq C_I \Vert v\Vert^{(2|I|-n)\epsilon} \prod_{i\in I^c}|v_{j(i)}|^2\leq
C_I\Vert v\Vert^{(2|I|-n)\epsilon-2\frac{|I|}{n-1}},
\quad{\text{si $|I|<n$.}}$$

Si $I=\{1,\dots,n\}$, on utilise l'in\'egalit\'e de H\"older
pour obtenir la majoration
$$R_I(v)\leq
\big(\int_{U_I}\frac{1}{\prod_{i=1}^n|\langle f_i,z\rangle|^{3/2}}\big)^{2/3}
\big(\int_{U_I} \frac{1}{(1+\sum_{i=1}^n|v_i^{-1}z_i|)^{3(n+1)}}\big)^{1/3}.$$
Comme $\int_{D(0,R)}\frac{1}{|y|^{3/2}}=\pi R^{1/2}$,
il existe une constante $C_I$ telle que
$$\big(\int_{U_I}\frac{1}{\prod_{i=1}^n|\langle f_i,z\rangle|^{3/2}}\big)^{2/3}
\leq C_I\Vert v\Vert^{\frac{n}{3}\epsilon}$$
Par ailleurs, il existe $A$ tel que $U_I\subset D(0,A \Vert v\Vert^{\epsilon})^n$.
Soit $i$ tel que $|v_i|$ soit minimal (et donc \'egal \`a
$\Vert v^{-1}\Vert^{-1}$ et $\geq \Vert v\Vert^{\frac{1}{n-1}}$).
On peut majorer $\frac{1}{(1+\sum_{i=1}^n|v_i^{-1}z_i|)^{3(n+1)}}$
par $\frac{1}{(1+|v_i^{-1}z_i|)^{3(n+1)}}$ et l'int\'egrale sur
$U_I$ par une int\'egrale sur $\prod_{j\neq i}
D(0,A \Vert v\Vert^{\epsilon})\times\C$.
On en d\'eduit la majoration
$$\int_{U_I} \frac{1}{(1+\sum_{i=1}^n|v_i^{-1}z_i|)^{3(n+1)}}\leq
(\pi A^2 \Vert v\Vert^{2\epsilon})^{n-1}|v_i|^{2}\int_\C
\frac{1}{(1+|z|)^{3(n+1)}}.$$
En regroupant les majorations obtenues, on en tire
l'existence d'une constante $C'_I$ telle que
$$\big|\int_{U_I}\frac{\phi(z)}{\prod_{i=1}^n\langle f_i,vz\rangle}\big|\leq
C'_I\Vert v\Vert^{\frac{n}{3}\epsilon+\frac{2}{3}((n-1)\epsilon-\frac{1}{n-1})}.$$
En prenant $\epsilon>0$ assez petit pour que $(2|I|-n)\epsilon-\frac{2|I|}{n-1}<0$,
pour tout $I\neq\{1,\dots,n\}$, et 
$\frac{n}{3}\epsilon+\frac{2}{3}((n-1)\epsilon-\frac{1}{n-1})<0$,
on en d\'eduit l'existence de $\delta>0$ et
$C'$ tels que
$\big|\int_{U_I}\frac{\phi(z)}{\prod_{i=1}^n\langle f_i,vz\rangle}\big|
\leq C' \Vert v\Vert^{-\delta}$,
pour tout $v\in V$ et tout $I\subset\{1,\dots,n\}$,
ce qui permet de d\'emontrer la convergence de la s\'erie
et termine la d\'emonstration du lemme.
\end{proof}

\Subsection{Le cas complexe}\label{SSS5}
\subsubsection{Le $n$-cocycle $B\mapsto g_B$}\label{SSS6}
Soit $X={\bf P}^{n-1}(\C)$. On note 
$$\partial_m:\Z[X^{m+1}]\to\Z[X^m]$$
 l'application
obtenue par $\Z$-lin\'earit\'e \`a partir
de 
$$(f_0,\dots,f_m)\mapsto\sum_{i=0}^{m}(-1)^i(f_0,\dots,f_{i-1},f_{i+1},\dots,f_m)$$
On a $\partial_{m-1}\circ\partial_{m}=0$, et ${\rm Ker}\,\partial_{m-1}={\rm Im}\,\partial_{m}$

Si $U\subset X^n$, on dit qu'une fonction $\Z$-lin\'eaire $g$ sur $\Z[U]$ est {\it un $n$-cocycle sur $U$}
si $g(\partial{\cal B})=0$ pour tout ${\cal B}\in\Z[X^{n+1}]$ tel que $\partial_n{\cal B}\in \Z[U]$.
Cela \'equivaut \`a ce que $g({\cal B})=0$ pour tout ${\cal B}\in\Z[U]$ v\'erifiant
$\partial_{n-1}{\cal B}=0$.

On note ${\cal U}_0\subset X^n$ l'ensemble
des $B=(f_1,\dots,f_n)$ formant une base de $\C^n$, et
${\cal U}\subset {\cal U}_0$ l'ensemble
des $B=(f_1,\dots,f_n)$ 
telles que toute sous-famille de $n$ vecteurs
de $e_1,\dots,e_n,f_1,\dots,f_n$ soit libre (cette condition
est \'equivalente \`a la non annulation de tous les mineurs
de la matrice d\'efinie par $B$).

Si $B=(f_1,\dots,f_n)\in{\cal U}_0$, on note $B^\vee=(f_1^\vee,\dots,f_n^\vee)\in{\cal U}_0$ sa duale
(on a ${\rm Tr}(f_i^\vee f_j)=0$ si $i\neq j$ et ${\rm Tr}(f_i^\vee f_j)\neq 0$ si $i= j$;
comme les $f_i$ sont des \'el\'ements de $X$, on les rel\`eve dans $\C^n$ pour donner
un sens \`a ${\rm Tr}(f_i^\vee f_j)$).
Par $\Z$-lin\'earit\'e, cela fournit une involution
${\cal B}\mapsto{\cal B}^\vee$ de $\Z[{\cal U}_0]$.

Si $B=(f_1,\dots,f_n)\in{\cal U}_0$, soit $g_B$ la fonction m\'eromorphe sur $\C^n$
d\'efinie par:
$$g_B(z):=\frac{\det B}{{\rm Tr}(f_1z)\cdots{\rm Tr}(f_nz)}$$
(Le r\'esultat ne d\'epend pas du choix de rel\`evements des $f_i$ dans $\C^n$.)
On \'etend $B\mapsto g_B$ par lin\'earit\'e \`a $\Z[{\cal U}_0]$:
$$g_{\cal B}=\sum n_Bg_B,\quad{\text{si ${\cal B}=\sum n_BB$.}}$$
\begin{lemm}\phantomsection\label{Shi1}
$B\mapsto g_B$, $B\mapsto g_{\overline B}$
 et $B\mapsto g_{B^\vee}$ sont des $n$-cocycles sur ${\cal U}_0$.
\end{lemm}
\begin{proof}
Pour $B\mapsto g_B$, cf.~\cite{Wn}:
la preuve consiste \`a d\'evelopper par rapport \`a la derni\`ere
ligne le d\'eterminant $(n+1)\times(n+1)$
o\`u la $i$-i\`eme colonne, pour $0\leq i\leq n$, est $f_i$ (pour les $n$ premi\`eres lignes)
et ${\rm Tr}(f_iz)$ pour la derni\`ere: ce d\'eterminant est $0$ puisque la derni\`ere
ligne est une combinaison lin\'eaire des $n$ premi\`eres.

Le cas de $B\mapsto g_{\overline B}$ s'en d\'eduit imm\'ediatement.
Enfin, la transform\'ee de Fourier de $\frac{1}{z}$ est $\frac{i}{\xi}$, et comme
$\langle\xi,z\rangle=\sum_{i=1}^n\langle\overline{f_i^\vee}\xi,f_iz\rangle$,
on en d\'eduit que la transform\'ee de Fourier de $g_B$ est
$$i^n|\det B|^{-2}\frac{\det B}{{\rm Tr}(\overline{f_1^\vee}\xi)\cdots{\rm Tr}(\overline{f^\vee_n}\xi)}=
i^n g_{\overline B^\vee}$$  
Il s'ensuit que $B\mapsto g_{\overline B^\vee}$ est un $n$-cocycle,
et comme $B\mapsto g_{\overline B}$ en est un aussi, il en est de m\^eme de
$B\mapsto g_{B^\vee}$.
\end{proof}
\begin{rema}\phantomsection\label{Shi2}
Pour prouver que $B\mapsto g_{B^\vee}$ est un $n$-cocycle,
on peut aussi v\'erifier directement
que, si $B=(f_0,\dots,f_n)$ est tel que
$\partial_n B\in\Z[{\cal U}_0]$, alors 
$\partial_{n-1}((\partial_n B)^\vee)=0$, et donc $(\partial_n B)^\vee=\partial_n B'$,
avec $B'\in\Z[X^{n+1}]$,
ce qui permet de d\'eduire le r\'esultat de celui pour $B\mapsto g_B$.
\end{rema}

Si ${\cal B}=\sum n_B B\in \Z[X^m]$
{\it le support $|{\cal B}|$} de ${\cal B}$ est l'ensemble des $f\in X$ apparaissant
comme \'el\'ement d'une $B$ avec $n_B\neq 0$.
On a $|\partial_{m-1}{\cal B}|\subset|{\cal B}|$ mais cette inclusion peut \^etre stricte
du fait de l'annulation de certains coefficients de la combinaison lin\'eaire.
Si ${\cal B}\in \Z[{\cal U}_0]$, les singularit\'es apparentes de $g_{\cal B}$
sont des p\^oles simples
le long des hyperplans d'\'equation ${\rm Tr}(fz)$, pour
$f\in|{\cal B}|$, certaines de ces singularit\'es peuvent n'\^etre qu'apparentes:
\begin{lemm}\phantomsection\label{Shi5}
Les singularit\'es de $g_{\cal B}$ sont des p\^oles simples
le long des hyperplans d'\'equation ${\rm Tr}(fz)$, pour
$f\in|\partial_{n-1}{\cal B}|$.
\end{lemm}
\begin{proof}
Si $f\in|{\cal B}|\moins|\partial_{n-1}{\cal B}|$, on peut fabriquer
${\cal B}'$ avec $\partial_{n-1}{\cal B}=\partial_{n-1}{\cal B}'$ en rempla\c{c}ant $f$ par $f'\neq f$,
suffisamment g\'en\'erique, dans tous les $B$ contenant $f$.  
La relation de $n$-cocycle implique alors que 
$g_{\cal B}=g_{{\cal B}'}$
et $g_{{\cal B}'}$ n'a pas de singularit\'e le long de ${\rm Tr}(fz)=0$.
\end{proof}

\subsubsection{D\'ecomposition de Shintani pour $\Tc$}\label{SSS7}
Soit $V$ un r\'eseau de $\Tc$.
Une {\it d\'ecomposition de Shintani de $\Tc$ modulo $V$}
est un
\'el\'ement ${\cal B}$ de $\Z[{\cal U}]$ tel que l'on ait
\begin{equation}\phantomsection\label{N1}
{1\over{{\rm N}(z)}}=\sum_{v\in V}g_{v{\cal B}}(z)
\end{equation}
d\'es que la s\'erie dans le second membre converge.
On note ${\rm Shin}(\Tc/V)$ l'ensemble des
d\'ecompositions de Shintani de $\Tc$ modulo~$V$.

\begin{rema}\phantomsection\label{Shi6}
(i) La s\'erie converge pour presque tout $z$ d'apr\`es le (i) du lemme~\ref{D20.3}.

(ii) Si ${\cal B}\in {\rm Shin}(\Tc/V)$, les singularit\'es le long
des ${\rm Tr}(f_{i,B}vz)$ sont apparentes; il s'ensuit (utiliser le (ii) du lemme~\ref{D20}) que
$\partial_{n-1}(\sum_{v\in V}v{\cal B})=0$, et si $U\subset \Tc$ est compact,
on peut regrouper les termes par paquets finis
pour supprimer les singularit\'es dans $U$, et alors la s\'erie
converge uniform\'ement sur $U$. La somme est donc holomorphe sur $U$ et comme elle
est \'egale \`a $\frac{1}{{\rm N}(z)}$ presque partout, elle vaut $\frac{1}{{\rm N}(z)}$
partout.

(iii) R\'eciproquement, si $\partial_{n-1}(\sum_{v\in V}v{\cal B})=0$, on peut localement regrouper
les termes de la s\'erie pour la rendre absolument convergente.  
La somme $F(z)=\sum_{v\in V}g_{v{\cal B}}(z)$ est alors
holomorphe sur $\Tc$ et ${\rm N}(z)F(z)$ se factorise \`a travers $\Tc/(\C^\dual\cdot V)$ qui est compact.
Il existe donc $c({\cal B})\in \C$ tel que $F(z)=c({\cal B})\frac{1}{{\rm N}(z)}$.
\end{rema}

\begin{rema}\phantomsection\label{S6}
 (i)  S'il existe $\delta>0$ tel que
$|{\rm Tr}(fvz)|\geq\delta$, pour tous $f\in{\rm Supp}({\cal B})$
et $v\in V$, la s\'erie dans le second membre converge
(en effet, comme $zf_{B,1},\dots,zf_{B,n}$ forment une base
de $\C^n$, il existe $c(z,B)>0$ tel que $\sup_{1\leq i\leq n}
|{\rm Tr}(f_{B,i}vz)|\geq c(z,B)\Vert v\Vert$, pour tout $v\in V$,
 et l'hypoth\`ese
implique que $\big|\frac{1}{\prod_{i=1}^n{\rm Tr}(f_{B,i}vz)}\big|
\leq \delta^{1-n}c(z,B)^{-1}\Vert v\Vert^{-1}$).

(ii)  Si $V\subset\Tp$ et si $\{(B,n_B)\}\in{\rm Shin}(\Tp/V)$,
alors 
$$\sum n_B\epsilon(\det B)B\in{\rm Shin}(\Tc/V)$$
En effet, on
peut \'ecrire, si $z=(z_1,\dots,z_n)$ avec ${\rm Re}\, z_i>0$,
$${1\over{{\rm N}(z)}}=\int_{\Tp} e^{-{\rm Tr}(xz)}\,dx=
\sum_{v\in V}\sum_{B}n_B
\int_{vC_B}e^{-{\rm Tr}(xz)}\,dx=\sum_{v\in V}\sum_{B}{{n_B\,|\det vB|}\over
{\prod_{i=1}^n {\rm Tr}(f_{B,i}vz)}},$$
ce qui nous donne le r\'esultat si ${\rm Re}\, z_i>0$.  Le cas g\'en\'eral
s'obtient par prolongement analytique.


(iii) On a $g_{v{\cal B}}(z)={\rm N}(v)g_{\cal B}(vz)$.  Il s'ensuit que
si $V\subset \Tco$, alors ${\cal B}$ est une d\'ecomposition
de Shintani modulo~$V$ si et seulement si
$\frac{1}{{\rm N}(z)}=\sum_{v\in V}g_{{\cal B}}(vz)$ pour tout
$z\in (\C^\dual)^n$ tel que la s\'erie dans le second membre converge.
C'est sous cette forme que le r\'esultat est utilis\'e dans
les applications aux valeurs sp\'eciales de fonctions~$L$.
\end{rema}
\begin{prop}\phantomsection\label{S7}
Si $V$ est un r\'eseau de $\Tc$, l'ensemble ${\rm Shin}(\Tc/V)$ 
des d\'ecompositions de Shintani de $\Tc$
modulo $V$ n'est pas vide.
\end{prop}
\begin{proof}
   Soit $\eta_1,\dots,\eta_{n-1}$ une base de $V$ sur $\Z$.
Si $\sigma\in S_{n-1}$, posons $f_{1,\sigma}=1$ et $f_{i,\sigma}=
\prod_{j<i}\eta_j$ si $2\leq i\leq n$. 
Soit $B_\sigma=(f_{1,\sigma},\dots,f_{n,\sigma})$ comme dans la rem.\,\ref{sigma}..
Nous allons d\'eduire de~\cite{Cz1} (cf.~aussi~\cite{Wn}) que,
si $B_\sigma\in{\cal U}$ pour tout $\sigma\in S_{n-1}$,
alors 
\begin{equation}\phantomsection\label{S7.3}
\sum_{\sigma\in S_{n-1}}\epsilon({\rm Log}\,B_\sigma)\,B_\sigma\in {\rm Shin}(\Tc/V)
\end{equation}
   Soit $\alpha_i$ une racine $n$-i\`eme de ${\rm N}(\eta_i)$.
Soient $\eta'_i=\alpha_i^{-1}\eta_i$ (o\`u $\alpha_i$ est
identifi\'e
avec l'\'el\'ement $(\alpha_i,\dots,\alpha_i)$ de $\Tc$) et
$V'$ le sous-groupe de $\Tc$ engendr\'e par $\eta_1,\dots,\eta_{n-1}$.
On v\'erifie facilement que les termes individuels, et donc a fortiori la somme,
de la s\'erie
$$\sum_{v\in V}\sum_{\sigma\in S_{n-1}}{{\epsilon
(\sigma)\det {vB_\sigma}}\over{\prod_{i=1}^n {\rm Tr}(vf_{i,\sigma}z)}}$$
ne changent pas quand on remplace $\eta_i$ par $\eta'_i$ (dans la d\'efinition
de $f_{i,\sigma}$) et $V$ par $V'$.  On peut donc supposer que $V\subset\Tco$;
mais alors, le r\'esultat est une
cons\'equence directe de~\cite[th.\,1]{Cz1}.
\end{proof}

\begin{prop}\phantomsection\label{Shi7}
Si ${\cal B}\in{\rm Shin}(\Tc/V)$, alors ${\cal B}^\vee\in {\rm Shin}(\Tc/V)$.
\end{prop}
\begin{proof}
Il r\'esulte de la preuve du lemme~\ref{D20.3} que la transform\'ee de Fourier
de $\sum_{v\in V} g_{v{\cal B}}$ est $i^n \sum_{v\in V} g_{\overline v\overline{\cal B}^\vee}$,
o\`u $\overline{\cal B}^\vee=\sum_B n_B\overline B^\vee$.
Si ${\cal B}\in{\rm Shin}(\Tc/V)$, cela se traduit par
$\frac{1}{N(y)}=\sum_{v\in V}g_{\overline v^{-1}\overline{\cal B}^\vee}(y)
=\sum_{v\in V}g_{\overline v\overline{\cal B}^\vee}(y)$.
Si on prend le conjugu\'e complexe des deux membres et que l'on \'evalue
en $\overline y$ au lieu de $y$, on obtient
$\frac{1}{N(y)}=\sum_{v\in V}g_{v{\cal B}^\vee}(y)$,
ce qui permet de conclure.
\end{proof}

\begin{rema}
(i)
Si on applique l'op\'erateur $\big(\prod_{i=1}^n\frac{\partial}{\partial z_i}\big)^{k-1}$
aux deux membres de~(\ref{N1}), on obtient une identit\'e du type
$\frac{1}{{\rm N}(z)^k}=\sum_{v\in V}g_{v{\cal B},k}(z)$, o\`u $g_{v{\cal B},k}$ est
une combinaison lin\'eaire de fonctions du type $\frac{1}{\ell_1^{k_1}\cdots \ell_n^{k_n}}$
o\`u les $\ell_i$ sont des formes lin\'eaires et les $k_i$ des entiers~$\geq 1$ v\'erifiant
$\sum_{i=1}^nk_i=nk$.

(ii) Pour \'eliminer les restrictions de~\cite[th.\,5]{Cz1} et du corollaire \`a~\cite[th.\,3]{CS},
il suffirait, si $0\leq j<k$, d'exhiber une ``d\'ecomposition de Shintani'' de la forme suivante:
$$C\,\frac{(\overline{z_1\cdots z_n})^j}{(z_1\cdots z_n)^k}=
\sum_{v\in V}\sum_B\alpha_B\prod_{i=1}^n\frac{\overline{({\rm Tr}(f_{B,i}vz)})^j}{({\rm Tr}(f_{B,i}vz))^k}$$
avec $C\neq 0$.  Si on prend des $B$ de la forme $(1,v_{B,1},\dots, v_{B,n-1})$, avec
$v_{B,i}\in V$, l'existence de $C$ est \'equivalente \`a ce que la singularit\'e le long
de ${\rm Tr}(z)=0$ ne soit qu'apparente (si c'est le cas, toutes les singularit\'es ne sont qu'apparentes
puisque leur ensemble est invariant par $V$ et contenu dans l'ensemble des hyperplans ${\rm Tr}(vz)=0$, $v\in V$).

On peut aussi \'ecrire, formellement, $\frac{\overline{({\rm Tr}(f_{B,i}vz)})^j}{({\rm Tr}(f_{B,i}vz))^k}$
comme l'int\'egrale de $e^{-{\rm Tr}(tz)} P_B(t,\overline z)$ sur un c\^one $C_B$ 
port\'e par les $f_{B,i}$ (cf.~(ii)
de la rem.\,\ref{S6}), o\`u $P_B$ 
est un polyn\^ome bihomog\`ene de degr\'e $nj$ en $\overline z$ et $n(k-1)$ en $t$.
La disparition de la singularit\'e sur ${\rm Tr}(z)=0$ se traduit alors par le
fait que $\partial_n (\sum_v\sum_B P_{vB}\otimes vB)=0$, et il suffit de v\'erifier que le coefficient
de tout simplexe $(1,f_1,\dots,f_{n-2})$ contenant $1$ est nul ce qui donne $c(k,j)$ conditions pour chaque tel
simplexe.

Si on se restreint aux $v_{B,i}$ de la forme $\eta_1^{k_1}\cdots\eta_{n-1}^{k_{n-1}}$, o\`u
$\eta_1,\dots,\eta_{n-1}$ est une base de $V$ et $|k_i|\leq M$, cela fournit $c_0 M^{n-1}$ $B$ possibles,
et seulement $c_1 M^{n-2}$ simplexes $(1,f_1,\dots,f_{n-2})$ contenant $1$.  Si $M$ est assez grand,
$c_1M^{n-2}c(k,j)<c_0M^{n-1}$ et
il y a donc moins de conditions que de variables; on peut donc obtenir une identit\'e de
la forme voulue mais
sans la condition $C\neq 0$ indispensable pour pouvoir l'utiliser pour l'\'etude des fonctions $L$.
Prouver que l'on peut trouver une d\'ecomposition avec $C\neq 0$ revient \`a prouver que
l'\'equation $C=0$ est lin\'eairement ind\'ependante des \'equations d'annulation des
coefficients des simplexes, et toutes mes tentatives en ce sens ont \'echou\'e (mais je
n'ai aucun doute que le r\'esultat est correct).  Bergeron, Charollois, Garc\'{\i}a~\cite{BCG} ont trouv\'e
une voie alternative pour \'eliminer les restrictions de~\cite{Cz1,CS} mentionn\'ees ci-dessus. 
\end{rema}

\Subsection{Le cas r\'eel}\label{SSS8}

\subsubsection{La fonction $\chi_B$}\label{SSS9}
Soit $B=(f_1,\dots,f_n)\in{\cal U}_0(\R)$.
Si $y\in\Tr$ v\'erifie ${\rm Tr}(f_iy)\neq 0$ pour tout $1\leq i\leq n$,
soient 
$$\epsilon(f_i,y)=
\epsilon({\rm Tr}(f_iy)),
\quad f_i(y)=\epsilon(f_i,y)f_i, \quad
\epsilon(B,y)=\epsilon(\det B)\prod_{i=1}^n\epsilon(f_i,y)$$
et soit 
$$\chi_B^0(x,y)=\epsilon(B,y){\bf 1}_{C^0_B(y)}$$
 o\`u $C^0_B(y)$ est le c\^one ouvert engendr\'e par les
$f_i(y)$, pour $1\leq i\leq n$.

\begin{lemm}\phantomsection\label{R1}
Soient $B\in{\cal U}$ et
$$g^\vee_{B,y}(x)=\int_{{\bf R}^n}e^{2i\pi {\rm Tr}(tx)}g_B(2i\pi t+y) dt,$$
la transform\'ee de Fourier de $t\mapsto g_B(2i\pi t+y)$.
Alors 
$$g^\vee_{B,y}(x)=\chi_B^0(x,y) e^{-{\rm Tr}(xy)}$$
si ${\rm Tr}(f_{j}^\vee x)\, {\rm Tr}(f_{j} y)\neq 0$, pour tout $1\leq j\leq n$.
\end{lemm}
\begin{proof}
   Un changement lin\'eaire de variables
utilisant la formule
$${\rm Tr}(tx)=\sum_{j=1}^n{\rm Tr}(f_{j}t){\rm Tr}(f_{j}^\vee x)$$
permet d'\'ecrire
$$g^\vee_{B,y}(x)=\epsilon(\det B)\prod_{j=1}^n 
g({\rm Tr}(f_{j}^\vee x),{\rm Tr}(f_{j}y)),$$ o\`u l'on a pos\'e
$$g(x,y):=\int_{-\infty}^{+\infty}{{e^{2i\pi tx}}\over{2i\pi t+y}}dt
=\begin{cases} 0 &{\text{si $xy<0$,}}\\
\epsilon(y)e^{-xy}&{\text{si $xy>0$}}
\end{cases}$$
(L'int\'egrale s'\'evalue par la m\'ethode des r\'esidus
et vaut $\epsilon(y)2i\pi {\rm Res}_{t=-y/2i\pi}
{{e^{2i\pi tx}}\over{2i\pi t+y}}$ (resp.~$0$) si $xy>0$ (resp. $xy<0$).)
On en tire le r\'esultat.
\end{proof}


Maintenant, soit 
$$\chi_B(x,y):=\lim_{t\rightarrow 0^+}\chi_B^0(x+te_n,y)=
\epsilon(B,y){\bf 1}_{C_B(y)}$$
o\`u
$C_B(y)$ est le c\^one entrouvert obtenu \`a partir du c\^one
ouvert $C^0_B(y)$ comme au~\no\ref{eff1}.
\begin{rema}\phantomsection\label{partit}
(o) $\chi_B(\lambda x,\mu y)=\chi_B(x,y)$ si $\lambda,\mu>0$.

(i) Quand $y$ varie, les $C^0_B(y)$ prennent $2^n$ valeurs possibles (i.e.~les
$C^0(\pm f_1,\dots,\pm f_n)$ pour tous les choix de signes), qui forment une partition
de $\R^n$ aux hyperplans d'\'equation ${\rm Tr}(f_i^\vee x)=0$ pr\`es.
Il s'ensuit que les $C_B(y)$ prennent $2^n$ valeurs possibles, qui forment une partition
de $\R^n$. En effet, si $Y\subset\Tr$ est un ensemble de cardinal $2^n$, tel que
les $C^0_B(y)$, pour $y\in Y$, prennent toutes les valeurs possibles, alors
$\sum_{y\in Y}{\bf 1}_{C^0_B(y)}=1$ en dehors des hyperplans d'\'equation ${\rm Tr}(f_i^\vee x)=0$,
et donc, pour tout $x$, 
$$\sum_{y\in Y}{\bf 1}_{C_B(y)}(x)=\lim_{t\to 0^+}\sum_{y\in Y}{\bf 1}_{C^0_B(y)}(x+te_n)=1$$

(ii) $g_B$ ne change pas si on remplace $f_i$ par $-f_i$; il en est donc de m\^eme de $\chi_B$.
On peut donc ajuster les signes des $f_i$ de telle sorte que $e_n\in C(f_1,\dots,f_n)$.
On a alors
$$C_B(y)=\R(\epsilon(f_1,y))f_1+\cdots+\R(\epsilon(f_n,y))f_n$$
En particulier, $C_B(y)$ est ferm\'e si et seulement si ${\rm Tr}(f_iy)>0$ pour tout $i$.
Notons que {\og $C_B(y)$ ferm\'e\fg} \'equivaut \`a {\og $0\in C_B(y)$\fg}.
\end{rema}

\begin{lemm}\phantomsection\label{chi0}
Si $e_n\in C(f_1,\dots,f_n)$, alors $\chi_B(0,y)=\epsilon(\det B){\bf 1}_{C(f_1^\vee,\dots,f_n^\vee)}$
\end{lemm}
\begin{proof}
Si $e_n\in C(f_1,\dots,f_n)$,
les propri\'et\'es suivantes sont \'equivalentes:

\quad $\bullet$ 
$\chi_B(0,y)\neq 0$,

\quad $\bullet$ $0\in C_B(y)$,

\quad $\bullet$ ${\rm Tr}(f_iy)>0$ pour $1\leq i\leq n$,

\quad $\bullet$ $y\in C(f_1^\vee,\dots,f_n^\vee)$.

On en d\'eduit le r\'esultat car 
{\og $\chi_B(0,y)\neq 0$ \fg} implique {\og $\chi_B(0,y)=\epsilon(B,y)$\fg} et
{\og ${\rm Tr}(f_iy)>0$ pour $1\leq i\leq n$\fg} implique
{\og $\epsilon(B,y)=\epsilon(\det B)$\fg}.
\end{proof}

\begin{lemm}\phantomsection\label{signe}
Si $v\in\Tr$ v\'erifie $|{\rm N}(v)|=1$ et $v_n>0$, alors 
$$\chi_{vB}(x,y)={\rm N}(v)\chi_B(v^{-1}x,vy)$$
\end{lemm}
\begin{proof}
Le r\'esultat est imm\'ediat pour $\chi^0$ car $(vf_i)(y)=v\,f_i(y)$
(le ${\rm N}(v)$ vient de ce que $\det vB={\rm N}(v)\det B$). 
Le r\'esultat pour $\chi$ s'en d\'eduit car
\begin{align*}
\lim_{t\to 0^+}\chi_B^0(v^{-1}(x+te_n),vy)&=
\lim_{t\to 0^+}\chi_B^0(v^{-1}x+tv_n^{-1}e_n,vy)\\
&= \lim_{t\to 0^+}\chi_B^0(v^{-1}x+te_n,vy)=\chi_B(v^{-1}x,vy)
\end{align*}
puisque $v_n>0$, par hypoth\`ese.
\end{proof}

\begin{lemm}\phantomsection\label{chi1}
$B\mapsto \chi_B$ est un $n$-cocycle sur ${\cal U}(\R)$.
\end{lemm}
\begin{proof}
$B\mapsto g_B$ est un $n$-cocycle, et donc $B\mapsto g_B^\vee$ aussi, et
le r\'esultat est une cons\'equence du lien entre $\chi_B$ et $g^\vee_B$.
\end{proof}

\Subsubsection{Support de $\chi_B$ et de ses translat\'es}\label{SSS10}
\begin{lemm}\phantomsection\label{R5}  
Soient $B$ une base de $\R^n$ et $x,y\in({\bf R}^\dual )^n$.
Si $B,B^\vee \in{\cal U}$, 
il existe $c(B)>0$, 
telle que si $v\in({\bf R}^\dual )^n$ v\'erifie $|{\rm N}(v)|=1$ et  
$\chi_{vB}(x,y)\neq 0$, alors
$$\Vert v\Vert<c(B)
\sup_{1\leq i\leq n}\frac{\Vert y\Vert^2}{|y_i|^2},
\quad{\rm ou}\quad
\Vert v^{-1}\Vert<c(B)
\sup_{1\leq i\leq n}\frac{\Vert x\Vert^2}{|x_i|^2}$$
\end{lemm}
\begin{proof}
Remarquons que $\chi_{vB}(x,y)$ ne change pas si on multiplie $x$ ou $y$ par $\lambda>0$;
on peut donc supposer $\Vert x\Vert=\Vert y\Vert$.

Soient $f_1,\dots,f_n$ les \'el\'ements de $B$ et $f_1^\vee,\dots,f_n^\vee$ ceux de $B^\vee$.
   On a $\chi_{vB}(x,y)=\chi_B(v^{-1}x,vy)$ et donc
$\chi_{vB}(x,y)\neq 0$ implique ${\rm Tr}(vf_iy){\rm Tr}(v^{-1}f_i^\vee x)\geq 0$
 pour
tout $1\leq i\leq n$.  Comme 
$$\sum_{i=1}^n {\rm Tr}(vf_iy){\rm Tr}(v^{-1}f_i^\vee x)={\rm Tr}(xy),$$
on d\'eduit, que si $1\leq i\leq n$, une
au moins des in\'egalit\'es $|{\rm Tr}(vf_iy)|\leq \Vert y\Vert$ et $|{\rm Tr}(v^{-1}f_i^\vee x)|
\leq \Vert x\Vert$ est v\'erifi\'ee.
Il existe alors une partition de $\{1,\dots,n\}$
en deux parties $I_1$ et $I_2$,
telles
que l'on ait $|{\rm Tr}(vf_iy)|<\Vert y\Vert$ pour tout $i\in I_1$
et $|{\rm Tr}(v^{-1}f_i^\vee x)|<\Vert x\Vert$ pour tout $i\in I_2$.  

Soient $J_1=\{i\ |\ |v_i|\geq \Vert v\Vert^{1\over2}\}$ et
$J_2=\{i\ |\ |v_i|^{-1}\geq \Vert v\Vert^{1\over2}\}$.
On a $J_1\cap J_2=\emptyset$ et  
donc, l'une au-moins des 2 in\'egalit\'es $|J_1|\leq |I_1|$ et $|J_2|\leq |I_2|$
est v\'erifi\'ee.  Supposons que l'on soit dans le premier cas et renum\'erotons
les $v_i$ et les $f_j$ de telle sorte que $J_1=\{1,\dots,r\}$ et $I_1=
\{1,\dots,s\}$ avec $r\leq s$.  

Soient $a_{j,1},\dots,a_{j,n}$ les
coordonn\'ees du vecteur $f_j$; on obtient alors
$$|a_{j,1}y_1v_1+\cdots+  a_{j,s}y_sv_s| \leq
\Vert y\Vert+|a_{j,s+1}y_{s+1}v_{s+1}|+\cdots+|a_{j,n}y_nv_n|,
\quad{\text{pour $1\leq j\leq s$.}}$$
Soit $\delta_0$ le sup. des valeurs absolues des coefficients
de $B$, et supposons $\Vert v\Vert\geq 1$.
On obtient alors, divisant chacune des in\'egalit\'es
du syst\`eme par $\Vert v\Vert$ et utilisant le fait que
$|v_i|\leq\Vert v\Vert^{1\over2}$ si $i\geq s+1$,
$$
|a_{j,1}y_1{{v_1}\over{\Vert v\Vert}}+\cdots+  a_{j,s}y_s{{v_s}\over{\Vert v\Vert}}| \leq
(1+(n-s)\delta_0){{\Vert y\Vert}\over {\sqrt{\Vert v\Vert}}},
\quad{\text{pour $1\leq j\leq s$.}}$$
On peut r\'esoudre le syst\`eme exprimant les $y_i{{v_i}\over{\Vert v\Vert}}$
en fonction des $\sum a_{j,i}y_i{{v_i}\over{\Vert v\Vert}}$
gr\^ace aux formules de Cramer.  Soit
$\delta_1$ le sup. de l'inverse de la valeur absolue des mineurs
de $B$ (qui existe car $B\in{\cal U}$ et donc aucun de
ces mineurs n'est nul);  on a alors
$$\sup_{1\leq i\leq s}{{|y_iv_i|}\over{\Vert v\Vert}}\leq
\delta_0^{s-1}\delta_1(1+(n-s)\delta_0){{\Vert y\Vert}\over
{\sqrt{\Vert v\Vert}}},$$ d'o\`u l'on tire $\Vert v\Vert
\leq C\frac{\Vert y\Vert^2}{|y_i|^2}$, o\`u l'on a pos\'e
$C= (\delta_0^{s-1}\delta_1(1+n\delta_0))^2$,
et choisi $i$ tel que $|v_i|=\Vert v\Vert$.

 Dans le cas o\`u $|J_2|\leq |I_2|$, on obtient
une majoration du m\^eme type mais o\`u les r\^oles de $x$ et $y$
ont \'et\'e \'echang\'es et $B$ est remplac\'e par $B^\vee$ et $v$ par $v^{-1}$.  Cela permet de conclure.
\end{proof}

\subsubsection{D\'ecompositions de Shintani de $\Tr$}\label{SSS11}
Soit $V$ un r\'eseau de $\Tr$ tel que $v_n>0$ pour tout $v=(v_1,\dots,v_n)\in V$
et $v\mapsto v_n$ soit injectif sur $V$.

Une
{\it d\'ecomposition de Shintani de $\Tr$ modulo $V$} est un
\'el\'ement ${\cal B}=\sum n_BB$ de ${\bf Z}[{\cal U}(\R)]$ tel que, pour presque tout $y$, si
$x\in\Tr$, alors
\begin{equation}\label{somme}
\sum_{v\in V}{\rm N}(v)\chi_{v{\cal B}}(x,y)=\chi_{(e_1,\dots,e_n)}(x,y),
\quad{\text{o\`u $\chi_{\cal B}=\sum n_B\chi_B$.}}
\end{equation}
On note ${\rm Shin}(\Tr/V)$ l'ensemble des d\'ecompositions de Shintani de $\Tr$ modulo~$V$.

\begin{lemm}\phantomsection\label{R6}  
Si  $y\in\Tr$ est fix\'e, la s\'erie~{\rm(\ref{somme})}
est localement finie sur~$\Tr$.
\end{lemm}
\begin{proof}
Soit $a=(a_1,\dots,a_n)\in({\bf R}^\dual )^n$ et soit $U$ la
boule ouverte de centre $a$ et de rayon ${1\over2}\Vert a\Vert$.
Si $x\in U$, alors $\Vert x\Vert\leq{3\over2}\Vert a\Vert$ et
$|x_i|\geq\frac{1}{2}\inf_i|a_i|$.
On tire alors du lemme~\ref{R5} que
si $v\in V$ est tel qu'il existe $x\in U$ v\'erifiant
 $\chi_{v{\cal B}}(x,y)\neq0$, alors
$\inf(\Vert v\Vert,\Vert v^{-1}\Vert)\leq c(B,a,y)$.
L'ensemble des $v\in V$ v\'erifiant ceci \'etant fini, cela permet de
conclure.
\end{proof}
\begin{prop}\phantomsection\label{R2} On a
$${\rm Shin}(\Tr/V)=
\Z[{{\cal U}({\bf R})}]\cap {\rm Shin}(\Tc/V)$$
\end{prop}
\begin{proof}
Soit ${\cal B}\in\Z[{{\cal U}({\bf R})}]\cap{\rm Shin}(\Tc/V)$.
Comme $|g_{vB}(2i\pi t+y)|\leq |g_{vB}(y)|$, le lemme~\ref{D20.3}
assure que,
pour presque tout
$y\in{\bf R}^n$, la s\'erie de fonctions (de la variable $t$) $\sum_{v\in V}g_{v{\cal B}}(2i\pi t+y)$
converge normalement sur ${\bf R}^n$, et la somme vaut ${\rm N}(2i\pi t+y)^{-1}$
puisqu'on a suppos\'e ${\cal B}\in {\rm Shin}(\Tc/V)$.

La convergence normale assure la convergence dans l'espace
des distributions temp\'er\'ees
et on peut donc prendre
la transform\'ee de Fourier des deux 
membres, et multiplier par $e^{{\rm Tr}(xy)}$.  On obtient alors
$\sum_{v\in V}\chi_{v{\cal B}}(x,y)=\chi_{(e_1,\dots,e_n)}(x,y)$,
o\`u les deux membres sont consid\'er\'es comme des distributions (en $x$)
et $y$ est fix\'e (dans le compl\'ementaire de l'ensemble de mesure nulle
o\`u il n'y a pas convergence normale). Cela implique que les deux membres
sont \'egaux pour presque tout $x$.

Par ailleurs, $\sum_{v\in V}\partial_{n-1}v{\cal B}=0$ puisque ${\cal B}\in {\rm Shin}(\Tc/V)$.
On a donc, de m\^eme, $\sum_{v\in V}\partial_{n-1}v^{-1}{\cal B}^\vee=0$, et comme $\chi_{\cal B}$
est continue en dehors de discontinuit\'es le long des hyperplans d'\'equations
${\rm Tr}(fx)=0$, pour $f\in |\partial{\cal B}^\vee|$, on d\'eduit de la relation de cocycle
(comme dans le lemme~\ref{Shi5})
que $\sum_{v\in V}\chi_{v{\cal B}}(x,y)$ est continue sur $\Tr$. Comme
il en est de m\^eme de $\chi_{(e_1,\dots,e_n)}(x,y)$, l'\'egalit\'e presque partout entra\^{\i}ne
l'\'egalit\'e pour tout $x$ (\`a $y$ fix\'e, pour presque tout $y$), ce qui 
fournit l'inclusion du membre de droite dans le membre de gauche.

Pour d\'emontrer l'inclusion inverse, on multiplie l'identit\'e (\ref{somme}) par
$e^{-{\rm Tr}(xy)}$ et on prend la transform\'ee de Fourier inverse (la s\'erie
converge dans les distributions temp\'er\'ees d'apr\`es la discussion ci-dessus et
donc la s\'erie des transform\'ees de Fourier inverses aussi).  On en d\'eduit que,
pour presque tout $y$,  
$\frac{1}{{\rm N}(x+iy)}=\sum_{v\in V}g_{v{\cal B}}(x+iy)$ en tant que distributions
en $x$.  Mais on a $\sum_{v\in V}\partial_{n-1}v{\cal B}=0$, et on peut donc resommer
la s\'erie ci-dessus localement pour qu'elle devienne absolument convergente;
sa somme est alors holomorphe en $z=x+iy$, et donc \'egale \`a $\frac{1}{{\rm N}(z)}$
sur $\Tc$ puisque la diff\'erence est holomorphe et, pour presque tout $y$, nulle en
tant que distribution en $x$.  On en d\'eduit que
$\sum_{v\in V}g_{v{\cal B}}(z)=\frac{1}{{\rm N}(z)}$ d\`es que la s\'erie converge,
ce qui permet de conclure.
\end{proof}

\begin{coro}\phantomsection\label{R3.1}
Si ${\cal B}\in {\rm Shin}(\Tr/V)$, alors
${\cal B}^\vee\in {\rm Shin}(\Tr/V)$.
\end{coro}
\begin{proof}
Cela r\'esulte des prop.\,\ref{Shi7} et~\ref{R2}.
\end{proof}

\begin{coro}\phantomsection\label{R3}  
Si $\eta_1,\dots,\eta_{n-1}$ est une base de $V$ et si
$B_\sigma\in{\cal U}$ pour tout\footnote{ Voir le cas complexe
pour la d\'efinition de $B_\sigma$.} 
 $\sigma\in S_{n-1}$, alors
$$\sum_{\sigma\in S_{n-1}}\epsilon({\rm Log}\,B_\sigma)\,B_\sigma\in {\rm Shin}(\Tr/V)$$
\end{coro}
\begin{proof}
Cela r\'esulte de (\ref{S7.3}) et de la prop.\,\ref{R2}.
\end{proof}

\begin{rema}\phantomsection\label{R4}
(i) Si $V\subset\Tp$, et donc $\eta_1,\dots,\eta_{n-1}\in\Tp$, on a
$\chi_{B_\sigma}(x,y)=\epsilon(B_\sigma){\bf 1}_{C_{B_\sigma}}(x)$ si $y\in\Tp$.
Il s'ensuit que
$$\big\{\big(B_\sigma,\epsilon({\rm Log}\,B_\sigma)\epsilon(B_\sigma)\big),\sigma\in S_{n-1}\big\}
\in {\rm Shin}(\Tp/V)$$
ce qui fournit une preuve alternative du r\'esultat de~\cite{DF2} mentionn\'e dans la rem.\,\ref{sigma}.

(ii) Toujours si $V\subset\Tp$, 
$$\{(B,n_B)\}\in {\rm Shin}(\Tp/V)\Longrightarrow
\sum\epsilon(B)n_B\,B\in {\rm Shin}(\Tr/V)$$ 
(Cela suit du (ii) de la rem.\,\ref{S6}
et de la prop.\,\ref{R2}.)
\end{rema}

\section{M\'ethode de Shintani et s\'eries g\'en\'eratrices}
Soit $F$ un corps totalement r\'eel de degr\'e $[F:{\bf Q}]=n$.
Soient ${\cal O}_F$ son anneau d'entiers et $U_F$ le groupe des
unit\'es de ${\cal O}_F$.  Si $A$ est un anneau, soit ${\cal S}(F,A)$
l'ensemble des fonctions $\phi$ de $F$ dans $A$ qui sont localement
constantes et \`a support compact; localement constante signifiant
qu'il existe un id\'eal fractionnaire ${\goth a}$ de $F$ tel
que l'on ait $\phi(\alpha)=\phi(\beta)$ si $\alpha-\beta\in {\goth a}$ et \`a support
compact signifiant qu'il existe un id\'eal fractionnaire ${\goth b}$
de $F$ tel que l'on ait $\phi(\alpha)=0$ si $\alpha\notin{\goth b}$.

Soient $\tau_1,\dots,\tau_n$ les $n$ plongements de $F$ dans ${\bf R}$.
Consid\'erons $F$ comme un sous-${\bf Q}$-espace vectoriel de ${\bf R}^n$
en envoyant $\alpha\in F$ sur $(\alpha_1,\dots,\alpha_n)\in{\bf R}^n$
o\`u $\alpha_i=\tau_i(\alpha)$.  On peut alors consid\'erer les id\'eaux
fractionnaires de $F$ comme des r\'eseaux de ${\bf R}^n$.

Si $i\in\{1,\dots,n\}$, soit $\epsilon_i$ le caract\`ere de $({\bf R}^\dual)^n$
d\'efini par $\epsilon_i(x)=\epsilon(x_i)$ si $x=(x_1,\dots,x_n)\in({\bf R}^\dual)^n$.
Si $\psi=\prod_{i=1}^n\epsilon_i^{a_i}$ avec $a_i\in\Z/2\Z$,
soit ${\cal S}(F,\psi,A)$
le sous-espace de ${\cal S}(F,A)$ des fonctions $\phi$ v\'erifiant
$\phi(u\alpha)=\psi(u)\phi(\alpha)$ pour tout $u\in U_F$
et tout $\alpha\in F^\dual$.  

\Subsection{M\'ethode de Shintani}\label{SSS13}
On note ${\rm Shin}(F/V)$ l'ensemble des ${\cal B}=\sum n_BB\in{\rm Shin}(\Tr/V)$ 
dont le support $|{\cal B}|$ est inclus dans $F$. 
\begin{lemm}
${\rm Shin}(F/V)$ est non vide.
\end{lemm}
\begin{proof}
On peut, par exemple, utiliser le cor.\,\ref{R3} pour en construire des \'el\'ements explicites.
\end{proof}

On note ${\cal U}_0(F)$ et ${\cal U}(F)$ l'ensemble des $(f_1,\dots,f_n)\in {\cal U}_0$ ou ${\cal U}$
avec $f_i\in F^\dual \subset
\xymatrix@C=1pc{\C^n\moins\{0\}\ar@{->>}[r]& {\bf P}^{n-1}(\C)}$.
\subsubsection{Sommation de s\'eries g\'eom\'etriques}\label{SSS14}
Si $\alpha\in F$, posons
 $$q^{\alpha}:=e^{-{\rm Tr}(\alpha y)}$$
On a $q^\alpha q^\beta=q^{\alpha+\beta}$ 
si $\alpha,\beta\in F$.
Soit 
$$\Lambda_{F,A}:=A[q^\alpha,\tfrac{1}{1-q^\alpha},\,\alpha\in F]$$
\begin{prop}\phantomsection\label{D1}
Soient ${\cal B}\in\Z[{\cal U}_0(F)]$ et $\phi\in{\cal S}(F,A)$.

{\rm   (i)}  
La s\'erie $\sum_{\alpha\in F}\phi(\alpha)\chi_{\cal B}(\alpha,y)q^\alpha$
converge pour tout
$y\in{\bf R}^n$ tel que ${\rm Tr}(fy)\ne 0$ si $f\in|{\cal B}|$.

{\rm (ii)} Il existe $F(\phi,{\cal B},y)\in \Lambda_{F,A}$ tel que
$F(\phi,{\cal B},y)=
\sum_{\alpha\in F}\phi(\alpha)\chi_{\cal B}(\alpha,y)
q^\alpha$ d\`es que la s\'erie converge.
\end{prop}
\begin{proof}
Commen\c{c}ons par traiter le cas $F=\Q$ pour donner une id\'ee de ce qui se passe
dans le (ii).  On prend pour $B$ la base canonique $e_1$ de $\R$ (i.e.~$e_1=1\in\R$).
On a alors $C_B(y)=\R_{\geq 0}$ si $y>0$, et $C_B(y)=\R_{<0}$ si $y<0$, ainsi que
$\chi_B(-,y)={\bf 1}_{C_B(y)}$ si $y>0$ et $\chi_B(-,y)=-{\bf 1}_{C_B(y)}$ si $y<0$.

$\bullet$ Si $\phi={\bf 1}_\Z$, on obtient:
\begin{align*}
F(\phi,{\cal B},y)&=\sum\nolimits_{\alpha\geq 0}q^\alpha=\tfrac{1}{1-q}, \quad{\text{ si $y>0$,}}\\
F(\phi,{\cal B},y)&=-\sum\nolimits_{\alpha< 0}q^\alpha=\tfrac{-q^{-1}}{1-q^{-1}}=\tfrac{1}{1-q}, 
\quad{\text{ si $y<0$.}}
\end{align*}

$\bullet$ Plus g\'en\'eralement si $\phi$ est \`a support dans $\Z$ et p\'eriodique de p\'eriode $a$,
on obtient:
\begin{align*}
F(\phi,{\cal B},y)&=\sum\nolimits_{\alpha\geq 0}\phi(\alpha)q^\alpha=
\sum\nolimits_{\alpha=0}^{a-1}\phi(\alpha)\sum\nolimits_{n\geq 0}q^{\alpha+na}=
\sum\nolimits_{\alpha=0}^{a-1}\tfrac{\phi(\alpha){q^\alpha}}{1-q^a}, \quad{\text{ si $y>0$,}}\\
F(\phi,{\cal B},y)&=-\sum\nolimits_{\alpha< 0}\phi(\alpha)q^\alpha=
-\sum\nolimits_{\alpha=-a}^{-1}\phi(\alpha)\sum\nolimits_{n\geq 0}q^{\alpha-na}\\ &=
-\sum\nolimits_{\alpha=-a}^{-1}\tfrac{\phi(\alpha){q^\alpha}}{1-q^{-a}}
=\sum\nolimits_{\alpha=-a}^{-1}\tfrac{\phi(\alpha){q^{a+\alpha}}}{1-q^{a}}, \quad{\text{ si $y<0$.}}
\end{align*}
Mais $\phi$ \'etant p\'eriodique de p\'eriode $a$, on a
$\sum_{\alpha=-a}^{-1}{\phi(\alpha){q^{a+\alpha}}}=
\sum_{\alpha=0}^{a-1}{\phi(\alpha){q^{\alpha}}}$,
ce qui fait que l'expression est la m\^eme pour $y<0$ que pour $y>0$.

Apr\`es cet \'echauffement, passons au cas g\'en\'eral.
   Par lin\'earit\'e, on est ramen\'e \`a traiter le cas o\`u
${\cal B}=B$, avec $B=(f_1,\dots, f_n)$ et les
$f_i$ forment une base de $F$ sur~$\Q$.
Supposons $\phi$ constante modulo $\goth a$ et \`a support dans
$\goth b$ que l'on suppose contenir $f_1,\dots,f_n$,
et soit $a\in{\bf N}$ tel que $af_i\in\goth a$ pour tout
$1\leq i\leq r$.  

Posons $[0,1](+1)=[0,1[$ et $[0,1](-1)=]0,1]$.
Si $y\in{\bf R}^n$ est tel que ${\rm Tr}(f_iy)\neq0$ pour tout
$1\leq i\leq n$, soient $\epsilon_{y,i}\in\{\pm 1\}$ tels que
le c\^one $C_B(y)$ du \no\ref{SSS9} soit $\R(\epsilon_{y,1})f_1(y)+\cdots+\R(\epsilon_{y,n})f_n(y)$,
et soit 
$$X(a{B},y):=\{x\in{\bf R}^n\ |\ x=\sum\nolimits_{i=1}^n
x_iaf_i(y),\ {\rm avec}\ x_i\in[0,1](\epsilon_{y,i})\} $$
Tout \'el\'ement $x$ de $C_B(y)$ s'\'ecrit alors de mani\`ere unique sous la forme
$$x=x_0+\sum_{i=1}^n m_iaf_i(y),\quad{\text{avec $m_i\in{\bf N}$ et $x_0\in X(a{B},y)$.}}$$
On voit alors que, si ${\rm Tr}(f_iy)\neq0$ pour tout $1
\leq i\leq n$, la s\'erie
$\sum_{\alpha\in F}\phi(\alpha)\chi_{B}(\alpha,y)q^\alpha$
s'exprime en termes de s\'eries g\'eom\'etriques convergentes
et qu'elle vaut
$$\epsilon(B,y)\sum_{\alpha\in X(a{B},y)\cap{\goth b}}{{\phi(\alpha)
q^\alpha}\over
{\prod_{i\in I(y)}(q^{-af_i}-1)\prod_{i\notin I(y)}(1-q^{af_i})}}$$
o\`u
$I(y)$ d\'esigne l'ensemble des $i$ tels
que ${\rm Tr}(f_iy)<0$ (et donc $f_i(y)=-f_i$). 

Si $y_0\in {\bf R}^n$ v\'erifie
$I(y_0)=\emptyset$, alors $X(a{B},y)$ est le translat\'e de $X(a{B}, y_0)$ 
par $-\sum_{i\in I(y)}af_i$.  
On a donc
$$ X(a{B},y)\cap{\goth b}=\big\{\alpha-\sum\nolimits_{i\in I(y)}af_i,
\ \alpha \in X(a{B},y_0)\cap{\goth b}\big\}$$
Comme $\phi(\alpha+af_i)=\phi(\alpha)$ pour tout $1\leq i\leq n$,
on en d\'eduit que\footnote{Comme $I(y_0)=\emptyset$,
on a $\epsilon(B,y_0)=\epsilon(\det B)$.} 
\begin{equation}\label{F1}
\sum_{\alpha\in F}\phi(\alpha)\chi_{B}(\alpha,y)q^\alpha
=\epsilon(\det B)\sum_{\alpha\in X(a{B},y_0)\cap{\goth b}}{{\phi(\alpha)q^\alpha}\over
{\prod_{i=1}^n(1-q^{af_i})}}
\end{equation}
pour tout $y$ v\'erifiant ${\rm Tr}(f_iy)\neq0$ si $1\leq i\leq n$, ce qui
prouve le (ii).  
\end{proof}

\subsubsection{L'identit\'e fondamentale}\label{SSS15}
Soit
$$F^\dual(\phi,{\cal B},y)=F(\phi,{\cal B},y)-\phi(0)\chi_{\cal B}(0,y)$$
Posons $\beps(x)=(\epsilon_i(x))\in\{\pm1\}^n$, si $x\in\Tr$.

\begin{prop}\phantomsection\label{D15}
  Si $\phi\in{\cal S}(F,\psi,A)$,
et si ${\cal B}\in{\rm Shin}(F/V)$,
alors, pour presque tout $y\in\Tr$, 
$$\sum_{v\in V}\psi(v){\rm N}(v)F^\dual(\phi,{\cal B},vy)=
\epsilon({\rm N}(y))
\sum_{{\beta\in F^\dual }\atop{\beps(\beta)=\beps(y)}}
\phi(\beta)e^{-{\rm Tr}(\beta y)}.$$
\end{prop}
\begin{proof}
On obtient, en utilisant le lemme~\ref{D16} ci-dessous pour justifier
les interversions de sommes (on passe de la $1^{\text{\`ere}}$ \`a la $2^{\text{\`eme}}$ ligne
en utilisant les identit\'es $\phi(v\alpha)=\psi(v)\phi(\alpha)$
et $\chi_{v{\cal B}}(v\alpha,y)={\rm N}(v)\chi_{{\cal B}}(\alpha,y)$ (lemme~\ref{signe}),
et de la $3^{\text{\`eme}}$ \`a la $4^{\text{\`eme}}$ ligne par d\'efinition (formule~\ref{somme})),
\begin{align*}
\sum_{v\in V}\psi(v){\rm N}(v)F^\dual(\phi,{\cal B},vy)= &
\sum_{v\in V}\Bigl(\sum_{\alpha\in F^\dual }\phi(\alpha)\chi_{{\cal B}}(\alpha,vy)
e^{-{\rm Tr}(\alpha vy)}\psi(v){\rm N}(v)\Bigr)\\
= & \sum_{v\in V}\sum_{\alpha\in F^\dual }\phi(v\alpha)\chi_{v{\cal B}}(v\alpha,y)
e^{-{\rm Tr}(\alpha vy)}\\
= & \sum_{\beta\in F^\dual }\phi(\beta)
e^{-{\rm Tr}(\beta y)}
\sum_{v\in V}\chi_{v{\cal B}}(\beta,y)\\
= & \sum_{\beta\in F^\dual }\chi_{(e_1,\dots,e_n)}(\beta,y)\phi(\beta)
e^{-{\rm Tr}(\beta y)},
\end{align*}
d'o\`u le r\'esultat.
\end{proof}

\begin{lemm}\phantomsection\label{D16}
  Si ${\cal B}\in \Z[{\cal U}_0(F)]$, alors pour presque tout $y$, la s\'erie
$$\sum_{v\in V}\sum_{\alpha\in F^\dual }\phi(\alpha)\chi_{{\cal B}}(\alpha,vy)
e^{-{\rm Tr}(\alpha vy)}\psi(v){\rm N}(v)$$
est absolument convergente.
\end{lemm}
\begin{proof}
On peut supposer ${\cal B}=B$, avec $B=(f_1,\dots,f_n)$.
Il r\'esulte du lemme~\ref{D20} (utilis\'e pour $f_i^\vee y$ au lieu de $y$)
que, pour presque tout $y$, l'ensemble des $v$
tels qu'il existe $i$ avec $|{\rm Tr}(f_i^\vee v^{-1} y)|\leq \Vert v^{-1}\Vert^{1\over2}=
\Vert v^{-\frac{1}{2}}\Vert$
est fini.  Quitte \`a remplacer quelques termes de la somme par $0$, on peut donc supposer
que $|{\rm Tr}(f_i^\vee v^{-1} y)|\geq \Vert v^{-\frac{1}{2}}\Vert$ pour tout $i$, si
le terme de la somme est non nul.

Comme 
$\chi_{v{\cal B}}(\beta,y)\neq 0\Rightarrow {\rm Tr}(f_i^\vee v^{-1}y)
{\rm Tr}(f_i v\beta)\geq 0$ pour tout $i\in\{1,\dots,n\}$, on a alors
${\rm Tr}(\beta y)\geq \Vert v^{-\frac{1}{2}}\Vert\sum_{i=1}^n|{\rm Tr} (f_iv\beta)|$,
et il existe donc $c(B)>0$ tel que ${\rm Tr}(\beta y)
\geq c(B)\Vert v^{-\frac{1}{2}}\Vert\,\Vert v\beta\Vert$
si $\chi_{v{\cal B}}(\beta,y) \neq 0$. 
Par ailleurs, $\phi$ 
ne prend qu'un nombre fini de valeurs et il existe ${\goth b}$ et $M>0$
tels que $\phi(\beta)=0$ si $\beta\notin{\goth b}$ et
$|\phi(\beta)|\leq M$ pour tout $\beta\in {\goth b}$.
On obtient donc
\begin{align*}
S&:=\sum_{v\in V}\sum_{\alpha\in F^\dual }|\phi(\alpha)\chi_{{\cal B}}(\alpha,vy)
e^{-{\rm Tr}(\alpha vy)}\psi(v){\rm N}(v)|\\ 
&= \sum_{v\in V}\sum_{\beta\in F^\dual }|\chi_{v{\cal B}}(\beta,y)\phi(\beta)
e^{-{\rm Tr}(\beta y)}| \\
& \leq M \sum_{v\in V}\sum_{\beta\in{\goth b}\moins\{0\}}
\hskip-3mm\exp\bigl(-c(B)\Vert v^{-\frac{1}{2}} \Vert\Vert v\beta\Vert\bigr)
= M \sum_{v\in V}\sum_{\alpha\in{\goth b}\moins\{0\}}
\hskip-3mm\exp\bigl(-c(B)\Vert v^{-\frac{1}{2}} \Vert\Vert \alpha\Vert\bigr)
\end{align*}
et il n'est pas difficile de prouver que la derni\`ere s\'erie double converge
(par exemple en remarquant que $\Vert v^{-\frac{1}{2}} \Vert$ et $\Vert \alpha\Vert$ sont
minor\'ees par $c>0$ et donc que $\Vert v^{-\frac{1}{2}} \Vert\Vert \alpha\Vert\geq
\frac{c}{2}(\Vert v^{-\frac{1}{2}} \Vert+\Vert \alpha\Vert)$).
\end{proof}

\Subsection{${\cal B}$-r\'egularit\'e}
%
On dit que $\phi\in{\cal S}(F,A)$ est {\it ${\cal B}$-r\'eguli\`ere} si $F(\phi,{\cal B},y)$ est ${\cal C}^\infty$
sur $\R^n$.
Nous donnons ci-dessous des mani\`eres de modifier minimalement $\phi$ pour la rendre ${\cal B}$-r\'eguli\`ere.
\begin{prop}\phantomsection\label{smoo}
Si $\phi$ est ${\cal B}$-r\'eguli\`ere,
on peut \'ecrire $F(\phi,{\cal B},y)$ sous la forme
$$F(\phi,{\cal B},y)=\sum_{B,\zeta}\frac{\alpha_{B,\zeta}}
{(1-\zeta_{1}q^{f_{1}})\cdots(1-\zeta_{r}q^{f_{r}})}$$
o\`u: 

$\bullet$ $r\leq n$ d\'epend de $(B,\zeta)=\{(f_i,\zeta_i),\,1\leq i\leq r\}$, 

$\bullet$ 
$f_{1},\dots,f_{r}$ sont des \'el\'ements de $F$ lin\'eairement ind\'ependants sur $\Q$,

$\bullet$ $\zeta_1,\dots,\zeta_r$ sont des racines de l'unit\'e~$\neq 1$,

$\bullet$ $\alpha_{B,\zeta}\in \C$. 
\end{prop}
\begin{proof}
On part de la formule~(\ref{F1}).  Si $\alpha\in X(aB,y_0)\cap{\goth b}$, on peut \'ecrire
$\alpha$ sous la forme $\alpha=\sum_{i=1}^na_i(af_i)$ avec $a_i\in\Q\cap[0,1[$. Si $a_i=\frac{b_i}{d_i}$,
et si $f'_i=\frac{a}{d_i}f_i$, on a alors
$\frac{q^\alpha}{\prod_{i=1}^n(1-q^{af_i})}=\prod_{i=1}^n\frac{q^{b_i}{f'_i}}{1-q^{d_if'_i}}$
avec $b_i<d_i$ entiers.  Une d\'ecomposition en \'el\'ements simples
permet d'\'ecrire $\frac{q^{b_i}{f'_i}}{1-q^{d_if'_i}}=\sum_{\zeta^{d_i}=1}\frac{\lambda_\zeta}{1-\zeta q^{f'_i}}$.
On en d\'eduit que $F(\phi,{\cal B},y)$ est une combinaison lin\'eaire de termes du type
$\prod_{i=1}^n\frac{1}{(1-\zeta_i q^{g_i})}$ o\`u les $\zeta_i$ sont des racines de l'unit\'e
et les $g_i$ des \'el\'ements de $F$ formant une base de $F$ sur $\Q$.


Si $B=(f_1,\dots,f_r)$ et $\zeta=(\zeta_1,\dots,\zeta_r)$, 
soit $G_{B,\zeta}=\prod_{i=1}^r\frac{1}{1-\zeta_iq^{f_i}}$.
Si $G=\sum_{B,\zeta}\alpha_{B,\zeta} G_{B,\zeta}$, on note $S(G)$ le support de $G$, i.e~l'ensemble
des $f\in F^\dual$ tel qu'il existe $(B,\zeta)$ avec $\alpha_{B,\zeta}\neq 0$ et $f\in B$.
On note aussi $S_\infty(G)$ l'ensemble des $f$ tels que
${\rm Tr}\,fy=0$ soit un p\^ole apparent de $G$, i.e.~l'ensemble
des $f\in F^\dual$ tel qu'il existe $(B,\zeta)=\{(f_i,\zeta_i),\ 1\leq i\leq r\}$ 
avec $\alpha_{B,\zeta}\neq 0$, et $i\in\{1,\dots,r\}$
avec $f=f_i$ et $\zeta_i=1$
(i.e.~un des termes du d\'enominateurs de $G_{B,\zeta}$ est $1-q^f$).

La discussion ci-dessus fournit une \'ecriture de $F(\phi,{\cal B},y)$ sous la forme
$F(\phi,{\cal B},y)=G^{(0)}$ avec $G^{(0)}=\sum_{B,\zeta}\alpha_{B,\zeta} G_{B,\zeta}$.
Choisissons des $\Q$-formes lin\'eaires $e_1^\dual,\dots,e_n^\dual$ sur $F$, qui prennent
des valeurs dans $\Z\moins\{0\}$ en
les \'el\'ements de $S(G^{(0)})$, et qui sont lin\'eairement ind\'ependantes. 

Soit $\lambda=M_1e_1^\dual+\cdots+M_ne_n^\dual$, o\`u $M_1\ll M_2\cdots\ll M_n$ sont des entiers~$\geq 1$.
Quitte \`a changer $\frac{1}{1-\zeta_iq^{f_i}}$ en $1-\frac{1}{1-\zeta_i^{-1}q^{-f_i}}$ et 
remplacer $G^{(0)}$ par l'expression obtenue en d\'eveloppant les produits,
on peut supposer que $\lambda(f)>0$ pour tout $f\in S(G^{(0)})$ (et donc $\lambda(f)$ est un entier~$\geq 1$).

Si $(B,\zeta)=\{(f_i,\zeta_i)\}$, et si $\xi_i^{\lambda(f_i)}=\zeta_i$,
on ne peut avoir
$\xi_i=\xi_j$ que si $\zeta_i=\zeta_j=1$. En effet, si $\zeta_i\neq 1$ est d'ordre $N_i$ et si
$\zeta_j$ est d'ordre $N_j$, alors $\xi_i$ est d'ordre $N_i\lambda(f_i)$ tandis que l'ordre
de $\xi_j$ divise $N_j\lambda(f_j)$ (et lui est \'egal si $N_j\neq 1$).  Il s'ensuit que,
si $\xi_i=\xi_j$, alors $N_i\lambda(f_i)$ divise $N_j\lambda(f_j)$.
Le
pgcd $(\lambda(f_i),\lambda(f_j))$ divise $e_n^\dual(f_{j})\lambda(f_{i})-
e_n^\dual(f_{i})\lambda(f_{j})=\sum_{k=1}^{n-1}
(e_n^\dual(f_{j})e_k^\dual(f_{i})-e_n^\dual(f_{i})e_k^\dual(f_{j}))M_k$ qui est non nul
car $f_i$ et $f_j$ sont non colin\'eaires et les $M_i$ ont \'et\'e choisis d'ordres de grandeur diff\'erents;
on en d\'eduit que $(\lambda(f_i),\lambda(f_j))$ est tr\`es petit par rapport \`a
$\lambda(f_i),\lambda(f_j)$ (qui sont d'ordre $M_n$).  On aboutit \`a une contradiction
car $\frac{\lambda(f_i)}{(\lambda(f_{i}),\lambda(f_{j}))}$ divise $N_j$, ce qui n'est pas possible puisque
le membre de gauche est beaucoup plus gros que le membre de droite.

Soit $W(G^{(0)})$ le sous-$\Q$-espace de $F$ engendr\'e par $S_\infty(G^{(0)})$.
Comme $e_n^\dual$ ne s'annule pas sur $S(G^{(0)})$, il en est de m\^eme de $\lambda$; en particulier,
$\lambda$ n'est pas identiquement nulle sur $W(G^{(0)})$; on choisit $f\in W(G^{(0)})$
avec $\lambda(f)=1$.  On peut alors \'ecrire tout $v\in S(G^{(0)})$ sous la forme $\lambda(v)f+v'$
avec $\lambda(v')=0$.
Une d\'ecomposition en \'el\'ements simples (en posant $X=q^f$, et donc $q^{f_i}=q^{f'_i}X^{\lambda(f_i)}$) 
fournit une identit\'e
$$\prod_{i=1}^r\frac{\lambda(f_i)}{1-\zeta_i q^{f_i}}=
\sum_{{\xi_k^{\lambda(f_k)}=\zeta_k}\atop{1\leq k\leq r}}
\sum_{i=1}^r\frac{1}{1-\xi_iq^{f'_i/\lambda(f_i)}q^f}
\prod_{j\neq i}\frac{1}{1-\xi_j\xi_i^{-1}q^{f'_j/\lambda(f_j)-f'_i/\lambda(f_i)}}$$
Notons que les exposants $f+\frac{f'_i}{\lambda(f_i)}=\frac{f_i}{\lambda(f_i)}$
et $\frac{f'_j}{\lambda(f_j)}-\frac{f'_i}{\lambda(f_i)}=\frac{f_j}{\lambda(f_j)}-\frac{f_i}{\lambda(f_i)}$
forment encore une famille libre sur $\Q$.

Comme les p\^oles de $F(\phi,{\cal B},y)$ ne sont qu'apparents, la somme (sur $(B,\zeta)$)
des termes faisant intervenir $\frac{1}{1-q^{f'_i/\lambda(f_i)}}q^f$ (i.e.~$\xi_i=0$) est identiquement nulle.
Si on \'elimine ces termes, on obtient une nouvelle expression $G^{(1)}$ pour $F(\phi,{\cal B},y)$
pour laquelle $S_\infty(G^{(1)})$ est inclus dans
$W(G^{(0)})\cap{\rm Ker}\,\lambda$ (en effet, d'apr\`es la discussion ci-dessus, $\xi_j\xi_i^{-1}$ ne peut
\^etre \'egal \`a $1$ que si $\zeta_i=\zeta_j=1$, et donc $f_i,f_j\in S_\infty(G^{(0)})\subset W(G^{(0)})$,
et $\frac{f'_j}{\lambda(f_j)}-\frac{f'_i}{\lambda(f_i)}=\frac{f_j}{\lambda(f_j)}-\frac{f_i}{\lambda(f_i)}
\in W(G^{(0)})$,
et par ailleurs $\frac{f'_j}{\lambda(f_j)}-\frac{f'_i}{\lambda(f_i)}\in{\rm Ker}\,\lambda$).
On a donc $W(G^{(1)})\subset W(G^{(0)})\cap{\rm Ker}\,\lambda$, et
apr\`es un nombre fini~$m\leq n$
d'it\'erations, on obtient $W(G^{(m)})=0$, ce qui signifie que 
$G^{(m)}$ est une \'ecriture de $F(\phi,{\cal B},y)$ sous la forme voulue.
\end{proof}

\vskip2mm
Soit $\phi\in {\cal S}(F,A)$ \`a support dans ${\goth b}$ et constante modulo~${\goth a}$ (quitte
\`a diminuer ${\goth a}$, on peut supposer que ${\goth a}\subset \O_F$, ce que nous ferons).
Soit ${\goth q}$ est un id\'eal premier non ramifi\'e de $\O_F$, de degr\'e~$1$ (i.e.~$\O_F/{\goth q}\cong
\Z/p$ o\`u $p={\rm N}{\goth q}$), et ne contenant pas ${\goth a}$. On d\'efinit
$\phi_{\goth q}$ par 
$$\phi_{\goth q}(\alpha)
=(1-{\rm N}{\goth q}{\bf 1}_{{\goth b}{\goth q}}(\alpha))\phi(\alpha)$$
Le lemme qui suit est adapt\'e de~\cite[lemme\,2b]{CN2} et de ses corollaires.
\begin{lemm}\phantomsection\label{smoo1}
Si $v_{\goth q}(f)=0$ pour tout $f\in|{\cal B}|$,
alors
$\phi_{\goth q}$ est ${\cal B}$-r\'eguli\`ere.
\end{lemm}
\begin{proof}
Comme dans la preuve de la prop.~\ref{D1}, on se ram\`ene au cas ${\cal B}=B$ avec $B=(f_1,\dots,f_n)$.

Si ${\goth c}$ est un id\'eal fractionnaire de $F$, on note ${\goth c}^\vee$ son
dual, i.e. ${\goth c}^\vee={\goth d}_F^{-1}{\goth c}^{-1}$.
On a 
$$({\bf 1}_{\goth b}-{\rm N}{\goth q}{\bf 1}_{{\goth b}{\goth q}})(\alpha)=-
\hskip-5mm \sum_{{\beta\in{\goth q}^{-1}{\goth b}^\vee/{\goth b}^\vee}\atop{\beta\neq 0}}
\zeta^{\beta\alpha}{\bf 1}_{\goth b}(\alpha)$$
En reprenant les calculs menant \`a la formule~(\ref{F1}), on obtient
\begin{equation}\label{F5}
F(\phi_{\goth q},{\cal B},y)=-\epsilon(\det B)\hskip-5mm
\sum_{{\beta\in{\goth q}^{-1}{\goth b}^\vee/{\goth b}^\vee}\atop{\beta\neq 0}}
\sum_{\alpha\in X(aB,y_0)\cap{\goth b}}\frac{\phi(\alpha)\zeta^{\beta\alpha}q^\alpha}
{\prod_{i=1}^n(1-\zeta^{af_i\beta}q^{af_i})}
\end{equation}
On en d\'eduit le r\'esultat car les hypoth\`eses sur ${\goth q}$ font que
$\zeta^{af_i\beta}\neq 1$ (c'est une racine primitive $p$-i\`eme de l'unit\'e) pour tout $i$,
si $\beta\in {\goth q}^{-1}{\goth b}^\vee\moins{\goth b}^\vee$.
\end{proof}

\begin{rema}\phantomsection\label{smoo2}
Il existe une infinit\'e de ${\goth q}$ v\'erifiant les conditions
du lemme~\ref{smoo1}
qui sont en plus principaux et admettent un g\'en\'erateur $\varpi$ totalement
positif congru \`a $1$ modulo~${\goth a}$. 
Pour un tel ${\goth q}$, on a
$$\phi_{\goth q}(\alpha)=\phi(\alpha)-N(\varpi)\phi(\frac{\alpha}{\varpi})$$
\end{rema}

\Subsection{Fonctions \`a d\'ecroissance suffisante}\label{SSS16}
\begin{defi} Une fonction $F(y)$ est {\it \`a d\'ecroissance
suffisante \`a l'infini}, si on peut trouver des formes lin\'eaires
$L_1,\dots,L_r$ en $y_1,\dots,y_n$ dont {\it aucun coefficient n'est~$0$}, et
des fonctions $F_1,\dots,F_r$ telles que $(1+|L_i(y)|)F_i(y)$ soit
born\'ee sur ${\bf R}^n$ et $F(y)=\sum_{i=1}^rF_i(y)$.
\end{defi}
\begin{lemm}\phantomsection\label{D11}
{\rm (i)}
Toute combinaison lin\'eaire de fonctions \`a d\'ecroissance
suffisante \`a l'infini l'est aussi.

{\rm (ii)} Si $F$ est \`a d\'ecroissance
suffisante \`a l'infini et $G$ est born\'ee sur ${\bf R}^n$, alors $FG$
est \`a d\'ecroissance suffisante \`a l'infini.

{\rm (iii)}  Si $F_1,\dots,F_r$ et $G_1,\dots,G_r$ sont born\'ees sur ${\bf R}^n$
et si $F_i-G_i$ est \`a d\'ecroissance suffisante \`a l'infini pour
tout $1\leq i\leq r$, alors $\prod_{i=1}^rF_i-\prod_{i=1}^rG_i$ est
\`a d\'ecroissance suffisante \`a l'infini.

{\rm (iv)}  Si $F:{\bf R}\rightarrow\C$ est \`a d\'ecroissance rapide
\`a l'infini et si $f\in F^\dual $, alors $F({\rm Tr}(fy))$ est \`a d\'ecroissance
suffisante \`a l'infini.
\end{lemm}
\begin{proof}
   les (i) et (ii)  sont imm\'ediats, le (iii)
se d\'emontre par r\'ecurrence sur $r$ en utilisant la formule $F_1F_2-G_1G_2=
(F_1-G_1)F_2+G_1(F_2-G_2)$ et le (iv) provient du fait que si
$f\in F^\dual $, alors $f$ n'a aucune de ses coordonn\'ees nulle.
\end{proof}

\subsubsection{D\'ecroissance suffisante de $F(\phi,{\cal B},y)$}\label{SSS17}
Le r\'esultat suivant est imm\'ediat.
\begin{lemm}\phantomsection\label{D10}  Soient $a\in [0,1]$ et $\zeta\neq1$ une 
racine de l'unit\'e;
posons
$$F_{a,\zeta}(y):={{e^{-ay}}\over{1-\zeta e^{-y}}}.$$

{\rm (i)}  Si $k\in{\bf N}$, la d\'eriv\'ee $k$-i\`eme $F^{(k)}_{a,\zeta}$
de $F_{a,\zeta}$ est born\'ee sur ${\bf R}$ et m\^eme \`a d\'ecroissance
rapide \`a l'infini si $k\geq 1$ ou si $k=0$ et $a\neq 0,1$.

{\rm (ii)}  Soient ${\bf 1}_+:={\bf 1}_{\R_+}$ et ${\bf 1}_-:={\bf 1}_{\R_-}$.
Alors $y\mapsto \zeta^aF_{a,\zeta}(y)-
(-1)^a{\bf 1}_{(-1)^a}(y)$ est \`a d\'ecroissance rapide
\`a l'infini si $a\in\{0,1\}$.
\end{lemm}

Soit $\partial_i$ l'op\'erateur diff\'erentiel ${{-\partial}\over{\partial y_i}}$.  
Si $k=(k_1,\dots,k_n)\in\N^n$, soit 
$$\bpart^k=\prod_{i=1}^n\partial_i^{k_i}$$

\begin{lemm}\phantomsection\label{D7}  
Soient ${\cal B}\in\Z[{\cal U}_0(F)]$, 
et $\phi\in{\cal S}(F)$, ${\cal B}$-r\'eguli\`ere.

{\rm (i)} 
$\F$ est une fonction ${\cal C}^\infty$ sur ${\bf R}^n$ dont toutes les
d\'eriv\'ees sont born\'ees sur ${\bf R}^n$. 

{\rm (ii)}  Plus pr\'ecis\'ement,
si ${k}\in{\bf N}^n-\{(0,\dots,0)\}$, alors
$\bpart^{k}F(\phi,{\cal B},y)$ est \`a d\'ecroissance suffisante
\`a l'infini.

{\rm (iii)}  $F(\phi,{\cal B},y)-\phi(0)\chi_{{\cal B}}(0,y)$
est \`a d\'ecroissance suffisante \`a l'infini.
\end{lemm}
\begin{proof}
Les (i) et (ii) sont une cons\'equence imm\'ediate de la prop.\,\ref{smoo}
et des lemmes~\ref{D10} et~\ref{D11}.

%
%

Pour d\'emontrer le (iii), on peut changer certains $f_i$ en $-f_i$ (ce qui ne change
pas $\chi_B$ et donc pas $F(\phi,{\cal B},y)$) pour assurer que $e_n\in C(f_1,\dots,f_n)$.
On peut aussi utiliser l'expression~(\ref{F1}) pour $F(\phi,{\cal B},y)$;
le seul terme qui n'est pas \`a d\'ecroissance suffisante (en faisant abstraction des
p\^oles apparents le long des hyperplans ${\rm Tr}(f_iy)=0$) est
celui correspondant \`a $\alpha=0$, qui vaut $\frac{\phi(0)}{\prod_{i=1}^n(1-q^{af_i})}$,
tend vers $\phi(0)$ \`a l'infini du c\^one d\'efini par
${\rm Tr}(f_iy)>0$ pour tout $i$ (i.e. $C(f_1^\vee,\dots,f_n^\vee)$) et est
\`a d\'ecroissance suffisante sur les c\^ones compl\'ementaires.
Plus pr\'ecis\'ement, on d\'eduit des lemmes~\ref{D10}
et~\ref{D11} (iii) que $F(\phi,{\cal B},y)-\phi(0){\bf 1}_{C(f_1^\vee,\dots,f_n^\vee)}$
est \`a d\'ecroissance suffisante.  On conclut en utilisant le lemme~\ref{chi0}.
\end{proof}


\section{Fonctions $L$ des corps totalement r\'eels}\label{SSS12}

Si $\phi\in{\cal S}(F,\psi,R)$ et si $V$ est
un sous-groupe d'indice fini de $U_F$, la s\'erie
$$L(\phi,\psi,s):={1\over{[U_F:V]}}\sum_{{\alpha\in F^\dual}\atop{\alpha\ {\rm mod}\ V}}{{\phi(\alpha)
\psi(\alpha)}\over{|{\rm N}(\alpha)|^s}}$$
ne d\'epend pas du choix de $V$,
converge absolument sur le demi-plan ${\rm Re}\, s>1$ et d\'efinit sur ce
demi-plan une fonction holomorphe.
\begin{rema}\phantomsection\label{smoo2.1}
Si ${\goth q}=(\varpi)$ comme dans la rem.\,\ref{smoo2}, on a
$$\phi_{\goth q}(\alpha)=\phi(\alpha)-N(\varpi)\phi(\tfrac{\alpha}{\varpi})
\quad{\rm et}\quad
L(\phi_{\goth q},\psi,s)=(1-N(\varpi)^{1-s})L(\phi,\psi,s)$$
ce qui permet de d\'eduire les propri\'et\'es de $L(\phi,\psi,s)$ 
pour $\phi$ g\'en\'erale, de celles dans le cas o\`u $\phi$ est ${\cal B}$-r\'eguli\`ere.
\end{rema}

Nous allons utiliser le raffinement de la m\'ethode de Shintani 
expos\'e dans la partie pr\'ec\'edente pour d\'emontrer que cette
fonction a un prolongement m\'eromorphe \`a tout le plan
complexe.  La m\'ethode est une petite variante de celle de
Shintani, mais fournit quelques informations suppl\'ementaires.
En particulier, elle permet 
de d\'emontrer directement que la fonction
$L(\phi,\psi,s)$ s'annule aux entiers n\'egatifs (\`a un ordre
d\'ependant explicitement de $\psi$) et de donner une formule pour le terme dominant
({\og directement\fg} signifiant que l'on n'utilise pas l'\'equation fonctionnelle
et la pr\'esence de p\^oles de la fonction $\Gamma$, mais que l'on
travaille directement en l'entier n\'egatif consid\'er\'e).

\Subsection{Une formule int\'egrale pour $L(\phi,\psi,s)$}\label{SSS18}
\Subsubsection{Transform\'ee de Mellin des fonctions \`a d\'ecroissance suffisante}\label{SSS19}
\begin{lemm}\phantomsection\label{D6} Si
$F$ est \`a d\'ecroissance suffisante \`a l'infini, alors

{\rm (i)}  l'int\'egrale
$\Lambda(F,\psi,s)={1\over{\Gamma(s)^n}}\int_{{\bf R}^n}F(y)\psi(y)\epsilon
({\rm N}(y))|{\rm N}(y)|^s\prod_{i=1}^n{{dy_i}\over{|y_i|}}$
converge absolument pour $0<{\rm Re}\, s<{1\over n}$.

{\rm (ii)}  Pour presque tout $x\in{\bf R}_+$, 
l'int\'egrale\footnote{$d^\dual y$ \'etant la mesure de Haar pour le groupe multiplicatif; i.e.
$d^\dual y=\prod_{i=1}^{n-1}{{dy_i}\over{|y_i|}}$.}
$G(F,x)=\int_{|{\rm N}(y)|=x}F(y)\psi(y)\epsilon({\rm N}(y))d^\dual y$ converge absolument

{\rm (iii)} L'int\'egrale
${1\over{\Gamma(s)^n}}\int_0^{+\infty}G(F,x)x^s{{dx}\over x}$ converge
absolument si $0<{\rm Re}\, s<{1\over n}$ et est \'egale \`a $\Lambda(F,\psi,s)$.
\end{lemm}
\begin{proof}
Remarquons
tout d'abord que, gr\^ace au th\'eor\`eme de Fubini,
 les (ii) et (iii) sont des cons\'equences imm\'ediates du (i)
et de la formule $d^\dual y{{dx}\over x}=\prod_{i=1}^n{{dy_i}\over{|y_i|}}$
si $x=|{\rm N}(y)|$.  Le (i) est une cons\'equence du lemme~\ref{D9} ci-dessous
et la d\'ecroissance suffisante de $F$.
\end{proof}
\begin{lemm}\phantomsection\label{D8}  
Soient $r,s\in{\bf R}$ v\'erifiant $s>0$, $r> 0$ et
$r+s<1$;  il existe alors une constante $C(r,s)$ telle que l'on ait
$$H(a,r,s)=\int_{-\infty}^{+\infty}|t|^{s-1}(1+|t+a|)^{r-1}dt\leq C(r,s)
(1+|a|)^{r+s-1},$$
pour tout $a\in{\bf R}$.
\end{lemm}
\begin{proof}
   On v\'erifie facilement que $H(a,r,s)$ est une fonction
continue de $a$ et on peut donc se contenter de prouver une telle
majoration pour $|a|>1$.  
On peut alors \'ecrire
$$H(a,r,s)
\leq\int_{-\infty}^{+\infty}|t|^{s-1}|t+a|^{r-1}\,dt
=|a|^{r+s-1}\int_{-\infty}^{+\infty}|t|^{s-1}|t+1|^{r-1}\,dt$$
d'o\`u l'on tire facilement le r\'esultat. 
\end{proof}
\begin{lemm}\phantomsection\label{D9}
  Soient $s_0,s_1,\dots,s_n\in{\bf R}$ v\'erifiant
$s_i>0$ et $\sum s_i<1$ et $L(y)=\sum_{i=1}^n \lambda_iy_i$, avec $\lambda_i
\neq0$ pour tout $1\leq i\leq n$;
alors 
$$\int_{{\bf R}^n}(1+|L(y)|)^{s_0-1}\prod_{i=1}^n |y_i|^{s_i-1}dy_i<\infty$$
\end{lemm}
\begin{proof}
   En utilisant le lemme~\ref{D8}, on peut
majorer cette int\'egrale par
$$C(s_0,s_n)|\lambda_n|^{-s_n}
\int_{{\bf R}^{n-1}}(1+|M(y)|)^{s_0+s_n-1}\prod_{i=1}^{n-1}
|y_i|^{s_i-1}dy_i,$$
o\`u $M(y)=\sum_{i=1}^{n-1}\lambda_iy_n$.  On d\'eduit
alors le r\'esultat par une r\'ecurrence imm\'ediate.
\end{proof}

\subsubsection{La fonction $G(\phi,{\cal B},x)$ et sa transform\'ee de Mellin}\label{SSS20}
On suppose que $\phi$ est ${\cal B}$-r\'eguli\`ere et
que $\phi(0)=0$ dans ce qui suit. La fonction $F(\phi,{\cal B},y)$ est
alors \`a d\'ecroissance suffisante \`a l'infini (lemme~\ref{D7})
et il r\'esulte du (i) du lemme~\ref{D6} que
$$G(\phi,{\cal B},x):=\int_{{|{\rm N}(y)|=x}} 
F(\phi,{\cal B},y)\psi(y)\epsilon({\rm N}(y))d^\dual y$$
converge pour presque tout $x>0$.
\begin{lemm}\phantomsection\label{D17}
  La fonction $G(\phi,{\cal B},x)x^{s-1}$ est sommable
si ${\rm Re}\, s>1$.  De plus, 
$$\int_{{\bf R}}G(\phi,{\cal B},x)x^{s-1}dx=\Gamma(s)^n[U_F:V]
L(\phi,\psi,s).$$
\end{lemm}
\begin{proof}
    Utilisant le (ii) du 
lemme~\ref{D6} pour intervertir $\sum_{v\in V}$ et $\int_{\{y,\,|{\rm N}(y)|=x\}/V}$,
puis la prop.\,\ref{D15}, on obtient:
\begin{align*}
G(\phi,{\cal B},x)=& 
\int_{\{y,\,|{\rm N}(y)|=x\}/V}\sum_{v\in V}F(\phi,
{\cal B},vy)\psi(vy)\epsilon({\rm N}(vy))d^\dual y\\
= &\int_{\{y,\,|{\rm N}(y)|=x\}/V}\sum_{{\beta\in F^\dual }\atop{\beps(\beta)=\beps
(y)}} \phi(\beta)e^{-{\rm Tr}(\beta y)}\psi(y)d^\dual y.
\end{align*}
On en tire (en remarquant que {\og $\beps(y)=\beps(\beta)$\fg} $\Rightarrow$
{\og $\psi(y)=\psi(\beta)$\fg})
\begin{align*}
\int_0^{+\infty}G(\phi,{\cal B},x)x^{s-1}dx
= & \int_{{\bf R}^n/V}
\sum_{\beps(\beta)=\beps(y)}\phi(\beta)e^{-{\rm Tr}(\beta y)}\psi(y)
|{\rm N}(y)|^{s-1}dy\\
= & {1\over{[V:V^+]}}\int_{{\bf R}^n/V^+}
\sum_{\beps(\beta)=\beps(y)}\phi(\beta)e^{-{\rm Tr}(\beta y)}\psi(\beta)
|{\rm N}(y)|^{s-1}dy\\
= & {1\over{[V:V^+]}}\int_{{\bf R}^n}
\sum_{{\beps(\beta)=\beps(y)}\atop{\beta\ {\rm mod}\ V^+}}
\phi(\beta)e^{-{\rm Tr}(\beta y)}\psi(\beta)
|{\rm N}(y)|^{s-1}dy \\
= & {{\Gamma(s)^n}\over{[V:V^+]}}
\sum_{\beps\in\{\pm1\}^n}
\sum_{{\beps(\beta)=\beps}\atop{\beta\ {\rm mod}\ V^+}}
\phi(\beta)|{\rm N}(\beta)|^{-s}\psi(\beta)\\
= &{\Gamma(s)^n}[U_F:V]L(\phi,\psi,s).
\end{align*}
Pour justifier ces calculs (c'est-\`a-dire pour d\'emontrer la sommabilit\'e),
il faut remplacer tous les termes ci-dessus par
leur valeur absolue et utiliser le fait que la derni\`ere s\'erie converge
absolument si ${\rm Re}\, s>1$.
\end{proof}
\begin{rema}
Il serait plus naturel de prouver simultan\'ement le lemme~\ref{D16} et la sommabilit\'e
de $G(\phi,{\cal B},x)x^{s-1}$ en prouvant que
$$\int_{\Tr}\sum_{v\in V}\sum_{\alpha\in F^\dual}|\phi(\alpha)||\chi_{\cal B}(\alpha,vy)|
e^{-{\rm Tr}(\alpha vy)}y^{\sigma-1}<\infty,\quad{\text{si $\sigma>1$.}}$$
  Malheureusement cette quantit\'e
n'est finie pour aucun $\sigma$, ce qui nous force \`a faire les contorsions ci-dessus.
\end{rema}

\begin{coro}\phantomsection\label{D18} 
{\rm (i)}  La fonction $G(\phi,{\cal B},x)x^{s-1}$ est sommable
si ${\rm Re}\, s>0$

{\rm (ii)}  La fonction $L(\phi,\psi,s)$ admet un prolongement analytique
au demi-plan ${\rm Re}\, s>0$ et, si $0<{\rm Re}(s)<\frac{1}{n}$,
$$L(\phi,\psi,s)={1\over{[U_F:V]}}{1\over{\Gamma(s)^n}}\int_{{\bf R}^n}
F(\phi,{\cal B},y)\psi(y)\epsilon({\rm N}(y))|{\rm N}(y)|^{s-1}dy$$
\end{coro}
\begin{proof}
   On a d\'emontr\'e que la fonction $G(\phi,{\cal B},x)x^{s-1}$
est sommable si ${\rm Re}\, s>1$ (lemme~\ref{D17}) ou si
$0<{\rm Re}\, s<{1\over n}$ ((iii) du lemme~\ref{D6} combin\'e au lemme~\ref{D7}).  
Maintenant, si ${1\over n}\leq{\rm Re}\, s
\leq 1$, alors $G(\phi,{\cal B},x)x^{s-1}$ est domin\'e par
$G(\phi,{\cal B},x)x$ si $x\geq1$ et par
$G(\phi,{\cal B},x)x^{{1\over{2n}}-1}$ si $x\leq 1$; d'o\`u le
(i).  

Le (ii) en r\'esulte car $\int_0^{+\infty}G(\phi,{\cal B},x)x^{s-1}$
d\'efinit une fonction analytique sur le demi-plan ${\rm Re}\, s>0$ valant
$[U_F:V]\Gamma(s)^nL(\phi,\psi,s)$ si ${\rm Re}\, s>1$ (lemme~\ref{D17}) et
$\int_{{\bf R}^n}
F(\phi,{\cal B},y)\psi(y)\epsilon({\rm N}(y))|{\rm N}(y)|^{s-1}dy,$
si $0<{\rm Re}\, s<{1\over n}$  (lemme~\ref{D6}). 
\end{proof}

\Subsection{Prolongement analytique des fonctions~$L$}\label{SSS21}
\subsubsection{Prolongement analytique de fonctions d\'efinies par des int\'egrales}\label{SSS22}
Si $I\subset \{1,\dots,n\}$ et $s=(s_1,\dots, s_n)\in\C^n$,
posons $s_I:=(s_i)_{i\in I}\in \C^I$.
Si $I^c:= \{1,\dots,n\}\moins I$, on a $s=(s_I,s_{I^c})$, et $s_I\mapsto (s_I,0_{I^c})$
fournit un plongement naturel $\C^I\hookrightarrow\C^n$.

Soit $H:({\bf R}_+)^n\rightarrow\C$ une fonction
${\cal C}^\infty$ telle que $\bpart^k H$ est \`a d\'ecroissance
suffisante \`a l'infini pour tout $k\in\N^n$.
\begin{lemm}\phantomsection\label{D21}  
  Si $I$ est une partie de $\{1,\dots,n\}$ et
$s_I\in\C^I$, alors l'int\'egrale 
$$\Lambda_I(H,s_I)=\prod_{i\in I}\tfrac{1}{\Gamma(s_i)}\int_{{\bf R}_+^I}
H(y)\prod_{i\in I} y_i^{s_i}\tfrac{dy_i}{y_i}$$
converge absolument si $0<{\rm Re}\, s_i< {1\over{n}}$ pour tout $i\in I$, 
et d\'efinit une fonction holomorphe de $s_I$
qui poss\`ede un prolongement analytique 
\`a $\{s_I,\ {\rm Re}\, s_i<{1\over n},\ {\text{si $i\in I$}}\}$. 
\end{lemm}
\begin{proof}
   Il suffit de d\'emontrer le r\'esultat pour $I=\{1,\dots,n\}$
car la restriction de $H$ \`a $\R_+^I$ est aussi une fonction ${\cal C}^\infty$
dont toutes les d\'eriv\'ees sont \`a d\'ecroissance suffisante \`a
l'infini.  

Soit $u:{\bf R}_+\rightarrow[0,1]$ une fonction
${\cal C}^\infty$ valant $1$ sur $[0,1]$ et $0$ sur $[2,+\infty[$.  Si
$J\subset \{1,\dots,n\}$, posons
$$H_J(y)=H(y)\Bigl(\prod_{j\in J}u(y_j)\Bigr)\Bigl(\prod_{j\notin J}
(1-u(y_j))\Bigr).$$
Alors $H=\sum_{J}H_J$ et donc $\Lambda_I(H,s)=
\sum_{J}\Lambda_I(H_J,s)$.
D'autre part, $\Lambda_I(H_J,s)$ est absolument convergente
sur l'ouvert $U_J$ produit des demi-plans
${\rm Re}\, s_i>0$ pour $j\in J$ et 
${\rm Re}\, s_j<{1\over n}$ pour $j\notin J$ (en particulier, 
elle s'annule si $s_i\in-{\bf N}$ pour un
$i\notin J$ \`a cause du z\'ero de $\Gamma(s_i)^{-1}$) et
une int\'egration par partie nous fournit l'identit\'e
\begin{equation}\label{partie}
\Lambda_I(H_J,s)=\Lambda_I(\bpart^{\ell_J}H_J,s_J+\ell_J,
s_{J^c})
\end{equation}
valable pour tout $\ell_J\in{\bf N}^J$.  On en
tire le prolongement analytique de $\Lambda_I(H_J,s)$ \`a l'ouvert $V_J:=\{
s\in\C^n,\  {\text{${\rm Re}\, s_j<{1\over n}$ pour
tout $j\notin J$}}\}$, et donc celui de
$\Lambda_I(H,s)$ \`a $\cap_J V_J=\{s,\ {\rm Re}\,s_j<\frac{1}{n}\ {\text{pour tout $j$}}\}$.
  Ceci permet de conclure.
\end{proof}

 Dans le cas o\`u $I=\{1,\dots,n\}$, la fonction
$\Lambda_I(H,s_I)$ est simplement not\'ee $\Lambda(H,s)$ ou
$\Lambda(H,s_1,\dots,s_n)$ suivant les cas.

\begin{lemm}\phantomsection\label{D21.1}
{\rm (i)}  Si $I\subset \{1,\dots,n\}$,
et si $k_{I^c}\in{\bf N}^{I^c}$, alors
$$\Lambda
(H,(s_I,-k_{I^c}))=\Lambda_I(\bpart^{k_{I^c}}H,s_I).$$

{\rm (ii)}  Soient $k=(k_1,\dots,k_n)\in{\bf N}^n$ et $I\subset\{1,\dots,n\}$.
On suppose que $\bpart^{k} H(y)$ 
est identiquement nulle sur chacun des hyperplans d'\'equation
$y_i=0$ pour $i\in I$; alors
$$\lim_{s\rightarrow -k}
\Bigl(\prod_{i\in I}\tfrac{1}{s_i+k_i}\Bigr)\Lambda(H,s)=
 \int_{{\bf R}_+^I}\bpart^{k} H(y_I,0_{I^c})\prod_{i\in I}\frac{d y_i}{y_i}$$
\end{lemm}
\begin{proof}
Pour d\'emontrer le (i),  
on d\'ecompose $H$ sous la forme $H=\sum_JH_J$ comme dans la preuve du lemme~\ref{D21}, et
on se ram\`ene par lin\'earit\'e \`a
$H=H_J$ et, par r\'ecurrence, au cas o\`u $I^c$ ne contient qu'un
\'el\'ement que l'on peut supposer \^etre \'egal \`a $n$ sans nuire
\`a la g\'en\'eralit\'e du raisonnement.  Il y a 2 cas.  

$\bullet$ Si
$J$ ne contient pas $n$, alors $\Lambda(H,(s_I,-k_n))$ est identiquement
nulle si $k_n\in{\bf N}$ comme on l'a vu au cours de la d\'emonstration du (i).
D'autre part, $H_J$ est identiquement nul dans un voisinage de ${\bf R}_+^I$
et donc $\partial_n^{k_n}H_J(y_I)$ est identiquement nulle; d'o\`u
l'\'egalit\'e dans ce cas l\`a.

$\bullet$ 
Si $n\in J$, la formule~(\ref{partie}) nous donne
$$\Lambda(H_J,(s_I,-k_n))=\Lambda(\partial_n^{k_n+1}H_J,(s_I,1))=
\prod_{i=1}^{n-1}\tfrac{1}{\Gamma(s_i)}\int_{{\bf R}_+^n}\partial_n^{k_n+1}
H_J(y)\Bigl(\prod_{i=1}^{n-1}y_i^{s_i}{\tfrac{dy_i}{y_i}}\Bigr)dy_n.$$
Sur le produit des demi-plans ${\rm Re}\, s_i>0$ (resp. ${\rm Re}\, s_i<{1\over n}$)
pour $i\in J$ (resp. $i\notin J$), l'int\'egrale est absolument convergente
et on peut commencer par int\'egrer par rapport \`a $y_n$ et l'on obtient
\begin{align*}
\Lambda(H_J,(s_I,-k_n))= &
\prod_{i=1}^{n-1}\tfrac{1}{\Gamma(s_i)}\int_{{\bf R}_+^{n-1}}\Bigl(\int_0^{+\infty}
\partial_n^{k_n+1}
H_J(y_I,y_n)dy_n\Bigr)
\bigl(\prod_{i=1}^{n-1}y_i^{s_i}{\tfrac{dy_i}{y_i}}\bigr)\\
= & \prod_{i=1}^{n-1}\tfrac{1}{\Gamma(s_i)}\int_{{\bf R}_+^{n-1}}
\partial_n^{k_n}
H_J(y_I,0)
\bigl(\prod_{i=1}^{n-1}y_i^{s_i}{\tfrac{dy_i}{y_i}}\bigr)
\end{align*}
ce que l'on voulait.

Le (ii) s'obtient apr\`es une int\'egration par partie utilisant la formule~(\ref{partie})
pour se ramener au cas $k=(0,\dots,0)$, et le (i) pour se d\'ebarrasser des $y_i$, pour $i\in I^c$.
\end{proof}
\begin{rema}\phantomsection\label{D21.5}
Dans la preuve des lemmes~\ref{D21} et~\ref{D21.1}, 
on n'utilise l'int\'egration par partie pour \'etablir l'existence
d'un prolongement analytique que pour $J\neq\emptyset$. On peut donc affaiblir un peu les conditions
et demander seulement que la restriction \`a
$[2,+\infty[^I$ (ou m\^eme $[c,+\infty[^I$, avec $c>0$) soit \`a d\'ecroissance suffisante
\`a l'infini, sans imposer de r\'egularit\'e (juste la mesurabilit\'e).  Par contre, on a vraiment
besoin que $H$ soit ${\cal C}^\infty$ au voisinage des hyperplans $y_i=0$, pour tout $i\in I$.
\end{rema}

\subsubsection{Application aux fonctions $L$}\label{SSS23}
Soit $\psi=\prod_{i=1}^n\epsilon_i^{a_i}$ avec $a_i
\in\Z/2\Z$ et soit $\phi\in
{\cal S}(F,\psi,A)$.

Soit $V$ un sous-groupe libre (c'est-\`a-dire
ne contenant pas $-1$) d'indice fini de
$U_F$ et ${\cal B}\in{\rm Shin}(F/V)$ une d\'ecomposition de Shintani modulo
$V$.  On suppose que $\phi$ est ${\cal B}$-r\'eguli\`ere.

\begin{theo}\phantomsection\label{D2}
{\rm   (i)}  La fonction $L(\phi,\psi,s)$ poss\`ede
un prolongement analytique \`a tout le plan complexe

{\rm (ii)}  
Si $k\in{\bf N}$, soit $I_k=\{i\in\{1,\dots,n\},\ 
a_i\equiv k\ [2]\}$.  Alors, si $k\neq0$ ou si
$\phi(0) =0$, la fonction $L(\phi,\psi,s)$ a un
z\'ero d'ordre $\geq |I_k|$ en $s=-k$ et
$$\lim_{s\rightarrow -k}(s+k)^{-|I_k|}L(\phi,\psi,s)=
{{2^{n-|I_k|}}\over{[U_F:V]}}\int_{{\bf R}^{I_k}}
\nabla^k F(\phi,{\cal B},y)\prod_{i\in I_k}{{dy_i}\over{y_i}},$$
o\`u $\nabla=\prod_{i=1}^n{\partial\over{\partial y_i}}$ et
${\bf R}^{I_k}$ est le sous-${\bf R}$-espace vectoriel de dimension
$|I_k|$ d'\'equation $y_i=0$ si $i\notin I_k$.
\end{theo}
\begin{proof}
D'apr\`es la formule int\'egrale du cor.\,\ref{D18}, valable pour ${\rm Re}(s)>0$,
$$L(\phi,\psi,s)={1\over{[U_F:V]}}{1\over{\Gamma(s)^n}}\int_{{\bf R}^n}
F(\phi,{\cal B},y)\psi(y)\epsilon({\rm N}(y))|{\rm N}(y)|^{s-1}dy$$
Pour d\'emontrer l'existence
d'un prolongement analytique au demi-plan ${\rm Re}\, s<{1\over n}$,
il suffit
d'appliquer le lemme~\ref{D21} \`a la fonction 
$$H(\phi,\psi,{\cal B},y)=
\sum_{\epsilon\in\{\pm1\}^n}\psi(\epsilon){\rm N}(\epsilon)F(\phi,{\cal B},
\epsilon y)$$ et \`a $(s_1,\dots,s_n)=(s,\dots,s)$.
Le reste de l'\'enonc\'e est alors une cons\'equence du lemme~\ref{D21.1}.
\end{proof}

\subsubsection{Le cas $\phi(0)\neq 0$}
On va maintenant s'int\'eresser au cas o\`u $\phi(0)\neq 0$, et au comportement
de $L(\phi,\psi,s)$ en $s=0$, si $\psi=1$ (si $\phi(0)=0$ il y a un z\'ero d'ordre $n$
en $s=0$ d'apr\`es le (ii) du th.\,\ref{D2}).

Si $V$ est un sous-groupe $V$ d'indice fini de~$U_F$,
on note ${\rm Reg}(V)$ le r\'egulateur de $V$, i.e.~le volume
d'un domaine fondamental de $\Tr\cap \{y,\ |{\rm N}(y)|=1\}$ modulo $V$
relativement \`a la mesure $d^\dual y:=\prod_{i=1}^{n-1}\frac{dy_i}{y_i}$.
On a alors ${\rm Reg}(U_F):=\frac{1}{[U_F:V]}{\rm Reg}(V)$ pour tout $V$.

\begin{theo}\phantomsection\label{residu}
Si $\phi(0)\neq 0$ et $\psi=1$, alors $L(\phi,\psi,s)$
a un z\'ero d'ordre $n-1$ en $s=0$, et on a
$$\lim_{s\to 0}s^{1-n}L(\phi,\psi,s)=-\phi(0)\,{\rm Reg}(U_F)$$ 
\end{theo}
\begin{proof}
On fixe $V$ d'indice fini dans~$U_F^+$.
Soit ${\cal B}_0\in{\rm Shin}(\Tp/V)$, et soit ${\cal C}=\sqcup_{B\in{\cal B}_0} C_B$, et donc
${\cal C}$ est un domaine fondamental de $\Tp$ modulo $V$, et
${\rm Reg}(V)$  est le volume de 
$${\cal C}^{=1}:={\cal C}\cap \{y,\ {\rm N}(y)=1\}$$
Maintenant, si $n_B=\epsilon(\det B)$, alors
$\sum_B n_B B\in{\rm Shin}(\Tc/V)$ (rem.\,\ref{S6} (ii)),
et donc $\sum_B n_B B^\vee\in{\rm Shin}(\Tc/V)$ (prop.\,\ref{Shi7}),
et ${\cal B}:=\sum_B n_B B^\vee\in{\rm Shin}(\Tr/V)$ (prop.\,\ref{R2}).

De plus, si $n_B\neq 0$, alors $e_n\in C(f_{B,1}^\vee,\dots,C_{B,n}^\vee)$
car ce c\^one est l'ensemble des $x$ tels que ${\rm Tr}(f_{B,i}x)>0$ pour tout $i$,
ce qui est le cas de $e_n$ puisque les $f_{B,i}$ sont \`a coordonn\'ees~$>0$.
Il s'ensuit que $\chi_{\cal B}(0,y)=\sum_B {\bf 1}_{C_B^0}={\bf 1}_{\cal C}$ presque partout.

Si $M>0$, soit ${\cal C}(M):={\cal C}\cap\{y\in\Tp,\ {\rm N}(y)\geq M\}$.
Le changement de variable $t=y_1\cdots y_n$, $z_1=t^{-1/n}y_1,\dots,z_n=t^{-1/n}y_n$ nous donne,
si ${\rm Re}(s)<0$,
$$\int_{{\cal C}(M)}{\rm N}(y)^s\prod_{i=1}^n\tfrac{dy_i}{y_i}=
\int_{t\geq M}\int_{z\in {\cal C}^{=1}}t^s\,d^\dual z\,\tfrac{dt}{t}=
-\tfrac{M^s}{s}{\rm Reg}(V)$$
qui a un prolongement m\'eromorphe \`a $\C$ tout entier.

D\'ecomposons $\F$ sous la forme 
$$\F=F^0(\phi,{\cal B},y) +\phi(0){\bf 1}_{{\cal C}(M)}$$
Alors $$H(y)=\sum_{\epsilon\in\{\pm1\}^n}N(\epsilon)F^0(\phi,{\cal B},\epsilon y)$$ est \`a
d\'ecroissance suffisante d'apr\`es le (iii) du lemme~\ref{D7}.
Le lemme~\ref{D21} (et la rem.\,\ref{D21.5}) montre alors que
$$\tfrac{1}{\Gamma(s)^n}\int_{\Tr}\F|{\rm N}(y)|^s\prod\tfrac{dy_i}{y_i}
:=\Lambda(H,s,\dots,s)-\phi(0){\rm Reg}(V)\tfrac{M^s}{s\Gamma(s)^n}$$
a un prolongement analytique \`a ${\rm Re}(s)<\frac{1}{n}$, et il
est facile de voir que
le r\'esultat ne d\'epend pas du choix de $M$.
On en d\'eduit que
$$\int_{\Tr}F(\phi,{\cal B},Ny)|{\rm N}(y)|^s\prod\tfrac{dy_i}{y_i}
=N^{-ns}\int_{\Tr}F(\phi,{\cal B},y)|{\rm N}(y)|^s\prod\tfrac{dy_i}{y_i}$$
D\'efinissons $\phi_N$ par $\phi_N(\alpha):=\phi(\alpha)-\phi(\alpha/N)$,
on a 
\begin{equation}\label{euler}
\phi_N(0)=0
\quad{\rm et}\quad
L(\phi_N,\psi,s)=(1-N^{-ns})L(\phi,\psi,s)
\end{equation}
Par ailleurs, 
le (o) de la rem.\,\ref{partit} nous donne
$$\sum_{\alpha}\phi(\tfrac{\alpha}{N})\chi_{\cal B}(\alpha,y)q^\alpha=
\sum_{\beta}\phi({\beta})\chi_{\cal B}(\beta,Ny)q^{N\beta}$$
puisque
$\chi_{\cal B}(N\beta,y)=\chi_{\cal B}(\beta,y)=\chi_{\cal B}(\beta,Ny)$.
Il s'ensuit que 
\begin{align*}
F(\phi_N,{\cal B},y)&=F(\phi,{\cal B},y)-F(\phi,{\cal B},Ny)\\
\int_{\Tr}F(\phi_N,{\cal B},y)|{\rm N}(y)|^s\prod\tfrac{dy_i}{y_i}
&=(1-N^{-ns})\int_{\Tr}F(\phi,{\cal B},y)|{\rm N}(y)|^s\prod\tfrac{dy_i}{y_i}
\end{align*}
o\`u le membre de gauche de la seconde ligne converge pour ${\rm Re}(s)>0$
d'apr\`es le (i) du cor.~\ref{D18} et admet un prolongement analytique \`a $\C$
d'apr\`es le lemme~\ref{D21}. On d\'eduit alors du (ii) du cor.~\ref{D21}
et de la formule~(\ref{euler}) que
$$[U_F:V]\,L(\phi,\psi,s)=\Lambda(H,s,\dots,s)-\phi(0){\rm Reg}(V)\tfrac{M^s}{s\Gamma(s)^n}$$
Comme $\Lambda(H,s,\dots,s)$ a un z\'ero d'ordre $n$ en $s=0$ d'apr\`es le (ii) du lemme~\ref{D21.1}
(et la rem.\,\ref{D21.5}),
il s'ensuit que $L(\phi,\psi,s)$ a un z\'ero d'ordre $n-1$ en $s=0$, de coefficient dominant
$\frac{1}{[U_F:V]}\phi(0){\rm Reg}(V)=
-\phi(0){\rm Reg}(U_F)$, ce qui permet de conclure.
\end{proof}

\Subsection{Valeurs sp\'eciales de fonctions~$L$ de Hecke}
Soit $\chi:\A_F^\dual/ F^\dual\to\C^\dual$ un caract\`ere de Dirichlet (i.e.~un caract\`ere
de Hecke d'ordre fini).  Si $v$ est une place finie de $F$, on note $\chi_v:F_v^\dual\to \C^\dual$
la restriction de $\chi$.  

$\bullet$ Si $v\mid\infty$, on a $F_v=\R$, et il existe $a_v\in\{0,1\}$ tel que
$\chi_v=\epsilon^{a_v}$.  

$\bullet$ Si $v$ est finie correspondant \`a un id\'eal premier ${\goth q}$, 
alors $\chi_v=1$ sur $\O_{\goth q}^\dual$ sauf
pour $v\in S(\chi)$, o\`u $S(\chi)$ est fini (et peut \^etre $\emptyset$).
Si $v\in S(\chi)$, il existe $n_v\geq 1$, minimal, tel que $\chi_v=1$ sur $1+{\goth q}^{n_v}$;
on pose ${\goth f}_\chi=\prod_{v\in S(\chi)}{\goth q}^{n_v}$ (c'est le conducteur de $\chi$).

Si $v\notin S(\chi)$ est finie, on pose $L_v(\chi,s)=(1-\chi_v(\varpi_v)({\rm N}v)^{-s})^{-1}$ o\`u
$\varpi_v$ est une uniformisante de $F_v$,
et on d\'efinit la fonction~$L$ (de Hecke) associ\'ee \`a $\chi$ par la formule
$L(\chi,s)=\prod_v L_v(\chi,s)$ (le produit \'etant sur les $v\notin S(\chi)$, finies).
Le produit converge absolument pour ${\rm Re}(s)>1$ et d\'efinit une fonction holomorphe
sur ce demi-plan.  

\subsubsection{La formule analytique du nombre de classes}
On dispose
du r\'esultat classique suivant (d\^u \`a Hecke~\cite{hecke2} \`a part pour le calcul du r\'esidu --- formule
analytique du nombre de classes --- qui est d\^u \`a Dedekind).
\begin{theo}\phantomsection\label{heck1}
{\rm (i)} Si $\chi\neq 1$, $L(\chi,s)$ poss\`ede un prolongement holomorphe \`a $\C$ tout entier.

{\rm (ii)} Si $\chi=1$, alors $L(\chi,s)=\zeta_F(s)$ poss\`ede un prolongement analytique \`a $\C$ tout entier,
holomorphe en dehors d'un p\^ole simple en $s=1$, 
de r\'esidu\footnote{Si $F=\Q$, on a $U_F=\{\pm1\}$ et ${\rm Reg}(U_F)=\frac{1}{2}$.}
 $2^n\frac{h_F{\rm Reg}(U_F)}{\sqrt{\Delta_F}}$.

{\rm (iii)} Si on pose\footnote{$\Gamma_\R(s)=\pi^{-s/2}\Gamma(s/2)$ o\`u $\Gamma$ est la fonction $\Gamma$
d'Euler.}
 $\Lambda(\chi,s)=\big(\prod_{v\mid\infty}\Gamma_\R(s+a_v)\big)L(\chi,s)$, alors il existe $N(\chi)\in\N$
et $w(\chi)\in\C^\dual$ tels que l'on ait l'\'equation fonctionnelle
$$\Lambda(\chi,s)=w(\chi)N(\chi)^{-s}\Lambda(\chi^{-1},1-s)$$
\end{theo}

\begin{rema}\phantomsection\label{heck2}
On d\'eduit de l'\'equation fonctionnelle du (iii)
 et du comportement de la fonction $\Gamma$ aux entiers n\'egatifs que:

$\bullet$ Si $k\geq 0$, si $I_k=\{v\mid\infty,\ a_v\equiv k\ [2]\}$, et si $\chi\neq 1$ ou si $k\neq 0$,
alors $L(\chi,s)$ a un z\'ero d'ordre $|I_k|$ en $s=-k$.

$\bullet$ $\zeta_F$ a un z\'ero d'ordre $n-1$ en $s=0$, et 
$\lim_{s\to 0}s^{1-n}\zeta(s)=-h_F\,{\rm Reg}(U_F)$.

Nous donnons ci-dessous (th.\,\ref{heck4}) une autre preuve de ces faits, sans utiliser l'\'equation fonctionnelle.
\end{rema}

On peut aussi voir $\chi$ comme un caract\`ere du groupe des id\'eaux fractionnaires de
$F$, \'etranger \`a $S(\chi)$: si ${\goth a}$ est un tel id\'eal, on pose
$\chi({\goth a})=\chi((\alpha_v)_v)$ o\`u $\alpha_v=1$ si $v\mid\infty$ ou si
$v\in S(\chi)$, et $v(\alpha_v)=v({\goth a})$ si $v\notin S(\chi)$ est finie.
  On a alors $L(\chi,s)=\sum_{{\goth a}\subset \O_F}\frac{\chi({\goth a})}
{{\rm N}{\goth a}^s}$.

Fixons un syst\`eme ${\goth a}_1,\dots{\goth a}_h$ de repr\'esentants du groupe des classes
d'id\'eaux; on impose que ${\goth a}_i\subset\O_F$ et ${\goth a}_i$ \'etranger \`a $S(\chi)$ pour tout $i$.
Alors tout id\'eal de $\O_F$ \'etranger \`a $S(\chi)$ s'\'ecrit sous la forme
${\goth a}_i(\alpha)$ avec $\alpha\in{\goth a}_i^{-1}$ \'etranger \`a $S(\chi)$,
et cette \'ecriture est unique \`a multiplication pr\`es de $\alpha$ par un \'el\'ement de $U_F$.
On a donc
$$L(\chi,s)=\sum_{i=1}^h\frac{\chi({\goth a_i})}
{{\rm N}{\goth a_i}^s}\sum'_{\alpha\in{\goth a}_i^{-1}/U_F}\frac{\chi((\alpha))}{|{\rm N}(\alpha)|^s}$$
Maintenant, si on revient \`a la d\'efinition de $\chi((\alpha))$ et que l'on utilise l'invariance
de $\chi$ par $F^\dual$, on voit que $\chi((\alpha))=\chi((\alpha'_v)_v)$
o\`u $\alpha'_v=\alpha^{-1}$ si $v\mid\infty$ ou si  
$v\in S(\chi)$, et $\alpha'_v=1$ sinon.  On a alors 
$$\chi((\alpha))=\prod_{v\mid\infty}\epsilon_v(\alpha)^{a_v}
\prod_{v\in S(\chi)}\chi_v^{-1}(\alpha)$$
 On en d\'eduit que
\begin{equation}\label{heck3}
L(\chi,s)=\sum_{i=1}^h\frac{\chi({\goth a_i})}
{{\rm N}{\goth a_i}^s} L(\phi_{\chi,{\goth a}_i},\psi,s)
\end{equation}
o\`u $\psi=\prod_{v\mid\infty}\epsilon_v^{a_v}$ et $\phi_{\chi,{\goth a}_i}(\alpha)=0$ si $\alpha\notin{\goth a}_i^{-1}$
ou si $(\alpha)$ n'est pas \'etranger \`a $S(\chi)$, et
$\phi_{\chi,{\goth a}_i}(\alpha)=\prod_{v\in S(\chi)}\chi_v^{-1}(\alpha)$ si $\alpha \in {\goth a}_i^{-1}$
et $(\alpha)$ est \'etranger \`a $S(\chi)$ (si $S(\chi)=\emptyset$ cette condition
est vide et le produit vaut $1$ comme tout produit vide qui se respecte). 
Notons que $\phi_{\chi,{\goth a}_i}$ est p\'eriodique de p\'eriode ${\goth f}_\chi$,
et que
$\phi_{\chi,{\goth a}_i}(0)=0$ sauf si $S(\chi)=\emptyset$ (o\`u l'on a $\phi_{\chi,{\goth a}_i}(0)=1$).

\begin{theo}\phantomsection\label{heck4}
{\rm (i)} $L(\chi,s)$ admet un prolongement analytique \`a $\C$ tout entier, holomorphe
en dehors d'un p\^ole simple \'eventuel en $s=1$ si $\chi$ est partout non ramifi\'e.

{\rm (ii)} Si $I_k=\{v\mid\infty,\ a_v\equiv k\ [2]\}$, et si $\chi\neq 1$ ou si $k\neq 0$,
alors $L(\chi,s)$ a un z\'ero d'ordre $\geq |I_k|$ en $s=-k$.

{\rm (iii)} $\zeta_F$ a un z\'ero d'ordre $n-1$ en $s=0$, et $\lim_{s\to 0}s^{1-n}\zeta(s)=-h_F\,{\rm Reg}(U_F)$.
\end{theo}
\begin{proof}
Compte-tenu de la formule~(\ref{heck3}), le (i) est une cons\'equence du (i) du th.\,\ref{D2}
et de la rem.\,\ref{smoo2.1} qui permet de localiser les p\^oles \'eventuels.  

Le (iii) r\'esulte du th.\,\ref{residu}, 
et le (ii) r\'esulte du (ii) du th.\,\ref{D2} 
sauf si $k=0$, $\chi$ est partout non ramifi\'e et $a_v=0$
pour tout $v\mid\infty$ o\`u il faut utiliser
le th.\,\ref{residu} et la nullit\'e de $\sum_{i=1}^h\chi({\goth a}_i)$
pour en d\'eduire que le coefficient du terme de degr\'e~$n-1$ 
s'annule et donc qu'il y a bien un z\'ero 
d'ordre~$\geq n$ en $s=0$.
\end{proof}

\begin{rema}\phantomsection\label{heck5}
On d\'eduit du
(ii) du th.\,\ref{D2}
une formule explicite pour $\lim_{s\to -k}(s+k)^{-|I_k|}L(\chi,s)$.  Il est possible
que l'on puisse exprimer le r\'esultat en termes de fonctions sp\'eciales \'evalu\'ees
en des \'el\'ements de $F$, mais il est peu probable que l'on puisse en extraire les
\'el\'ements dans les extensions ab\'eliennes de $F$ pr\'edits par les conjectures
de Stark et leurs avatars aux entiers~$<0$.
\end{rema}

\subsubsection{Int\'egralit\'e}
Soit $\phi_{\chi,{\goth a_i},{\goth q}}$ la fonction obtenue
\`a partir de $\phi_{\chi,{\goth a_i}}$ comme dans le lemme~\ref{smoo1}.
Comme les ${\goth a_i}{\goth q}^{-1}$ forment un
syst\`eme de repr\'esentants du groupes des classes d'id\'eaux, on a
\begin{equation}\label{heck10}
\sum_{i=1}^h\sum_{i=1}^h\frac{\chi({\goth a}_i)}{{\rm N}{\goth a}_i^s}
L(\phi_{\chi,{\goth a_i},{\goth q}},\psi,s)=(1-\chi({\goth q}){\rm N}{\goth q}^{1-s})L(\chi,s)
\end{equation}
\begin{rema}\phantomsection\label{heck11}
Si $\psi\neq \epsilon\circ{\rm N}^{k+1}$, on a $L(\chi,-k)=0$
d'apr\`es le (ii) du th.\,\ref{heck4} (sauf si $F=\Q$, $\chi=1$ et $k=0$, o\`u l'on a
$\zeta(0)=-\frac{1}{2}$).
\end{rema}

\begin{theo}\phantomsection\label{heck12}
Si $\psi= \epsilon\circ{\rm N}^{k+1}$, alors 
$$(1-\chi({\goth q}){\rm N}{\goth q}^{1+k})L(\chi,-k)\in 2^{n-1}\Z[\chi]$$
 pour tout ${\goth q}$ premier \`a ${\goth f}_\chi$.
\end{theo}
\begin{proof}
Choisissons un sous-groupe $V$ de $U_F$, libre, d'indice~$2$.
Choisissons aussi ${\cal B}\in{\rm Shin}(\Tr/V)$. Alors
$$(1-\chi({\goth q}){\rm N}{\goth q}^{1+k})L(\chi,-k)=
2^{n-1}\sum_{i=1}^h\chi({\goth a}_i){\rm N}{\goth a}_i^k
\nabla^k F(\phi_{\chi,{\goth a_i},{\goth q}},{\cal B},0)$$
D'apr\`es le (ii) du th.\,\ref{D2}.
Maintenant, il r\'esulte de la formule~(\ref{F5}) que, si $N{\goth q}=p$, alors
$\nabla^k F(\phi_{\chi,{\goth a_i},{\goth q}},{\cal B},0)={\rm Tr}_{\Q(\bmu_p,\chi)/\Q(\chi)}x$
avec $x\in \frac{1}{{\rm N}{\goth a}_i^k}\frac{1}{(\zeta_p-1)^{n(k+1)}}\Z[\bmu_p,\chi]$ 
(la puissance de $\zeta_p-1$ vient de ce que chaque d\'eriv\'ee augmente d'au plus~$1$
le degr\'e du d\'enominateur de la fraction rationnelle, et celle de $\frac{1}{{\rm N}{\goth a}_i^k}$
vient de ce que les $\alpha$ qui interviennent appartiennent \`a ${\goth a}_i^{-1}$, et appliquer
$\nabla$ \`a $e^{-{\rm Tr}(\alpha y)}$ donne un r\'esultat entier \`a multiplication
pr\`es par $\frac{1}{{\rm N}{\goth a}_i}$).
Les puissances de ${\rm N}{\goth a}_i$ se compensent,
et comme
${\rm Tr}_{\Q(\bmu_p,\chi)/\Q(\chi)}z\in\Z[\chi]$ si $z\in \frac{1}{(\zeta_p-1)^{n(k+1)}}\Z[\bmu_p,\chi]$
et si $p>1+n(k+1)$,
on en d\'eduit le r\'esultat pour tout ${\goth q}$ de degr\'e~$1$ sauf un nombre fini.

On conclut en utilisant le th\'eor\`eme de \v{C}ebotarev dont une cons\'equence est
que, si ${\goth q}$ est premier \`a ${\goth f}_\chi$, et si $N\geq 1$, il existe
${\goth q}'$ de degr\'e~$1$, avec $\chi({\goth q}')=\chi({\goth q})$ et ${\rm N}({\goth q}')\equiv
{\rm N}({\goth q})$ modulo $N$.
\end{proof}

\begin{rema}\phantomsection\label{heck13}
(i)
Deligne et Ribet~\cite{DR} ont prouv\'e que, si $\phi$ est \`a valeurs
dans~$\Z$, alors
$L(\phi_{\goth q},\psi,0)\in 2^{n}\Z$ sauf dans le cas suivant o\`u 
$L(\phi_{\goth q},\psi,0)\in 2^{n-1}\Z\moins 2^n\Z$:

$\bullet$ $U_F$ n'a pas d'\'el\'ement de norme~$1$ mais a des \'el\'ements de tous les signes possibles
compatibles avec cette restriction.

$\bullet$ $\phi(0)$ est impair.

$\bullet$ $\sigma_{\goth q}$ est non trivial dans ${\rm Gal}(K/F)$ o\`u $K$
est l'extension de $F$ obtenue en rajoutant les racines
carr\'ees des \'el\'ements de $U_F^+$ ($K$ est un corps totalement r\'eel et la
premi\`ere condition implique que $[K:F]=2$ ou $1$).

(ii) Gross~\cite[\S5]{G} a r\'einterpr\'et\'e ce r\'esultat en termes du r\'egulateur apparaissant
dans sa conjecture~\cite[conj.\,4.1]{G} sur les fonctions~$L$ en $s=0$ (version ``alg\`ebre de groupe'').
\end{rema}

\end{document}